\documentclass[12pt,oneside,openany, article]{memoir}

\nouppercaseheads

\usepackage[amsthm,thmmarks,hyperref]{ntheorem}

\usepackage[noTeX]{mmap}

\usepackage[english.us]{babel}
\usepackage[leqno]{mathtools}
\usepackage{empheq}

\usepackage{caption}

\usepackage[scr]{rsfso}
\usepackage{upgreek}

\usepackage[compress]{cite}
\usepackage{hypernat} 

\usepackage{fix-cm}
\usepackage[cm]{sfmath}
\usepackage{sansmathaccent}

\usepackage{microtype}
\newtheorem{theorem}{Theorem}[chapter]
\newtheorem{proposition}[theorem]{Proposition}

\newtheorem{corollary}[theorem]{Corollary}

\theoremstyle{remark}
\newtheorem{remark}[theorem]{Remark}
\newtheorem{conjecture}[theorem]{Conjecture}

\makeatletter
\newtheoremstyle{plainfoot}%
  {\item[\hskip\labelsep \theorem@headerfont ##1\ ##2\,\footnotemark\theorem@separator]}%
  {\item[\hskip\labelsep \theorem@headerfont ##1\ ##2\ (##3)\, \footnotemark\theorem@separator]}
\makeatother

\theoremstyle{plainfoot}
\theoremseparator{.}
\newtheorem{theorem-foot}[theorem]{Theorem}
\newtheorem{lemma-foot}[theorem]{Lemma}
\newtheorem{proposition-foot}[theorem]{Proposition}

\usepackage{subfig}
\usepackage{wrapfig}
\usepackage{array}

\usepackage{graphicx,xcolor}
\usepackage{eso-pic}
\usepackage[absolute,overlay]{textpos}

\usepackage{pdfpages}

\usepackage{todonotes}

\usepackage{xr-hyper}
\usepackage{hyperxmp}
\usepackage{zref, nameref}
\usepackage{url}
\usepackage[pdftex,bookmarks,pdfnewwindow,plainpages=false,unicode,pdfencoding=auto]{hyperref}

\usepackage{bookmark}
\usepackage[verbose]{placeins}
\usepackage{enumitem}
\usepackage{tikz}
\usepackage{pgfplots}

\usetikzlibrary{arrows, patterns,
decorations.pathreplacing, decorations.pathmorphing, shapes.symbols, spy}

\usepackage{relsize}

\usepackage{amssymb} 

\setsecheadstyle{\large\bfseries\raggedright}
\setsubsecheadstyle{\normalsize\bfseries\raggedright}

\newenvironment{claim}[1][{\textup{(\theequation)}}]{\refstepcounter{equation}\vglue10pt
\begin{trivlist}
\item[{\hskip\labelsep#1}]}{\vglue10pt\end{trivlist}}
\hypersetup{
colorlinks=true,
linkcolor=black,
citecolor=black,
urlcolor=blue,
pdfauthor={Victor Ivrii},
pdfauthortitle={Professor, Department of Mathematics, University of Toronto},
pdftitle={Pointwise Spectral Asymptotics  near Singularity},
pdfsubject={Sharp Spectral Asymptotics},
pdfkeywords={Microlocal Analysis,  Sharp Spectral Asymptotics},
pdfcopyright={Copyright (C) 2019 Victor Ivrii},
pdfcontactemail={ivrii@math.toronto.edu},
pdfcontacturl={http://www.math.toronto.edu/ivrii},
pdfcontactaddress={40 St. George St.},
pdfcontactcity={Toronto, ON},
pdfcontactpostcode={M5S 2E4},
pdfcontactcountry={Canada},
pdflang={en},
baseurl={http://www.math.toronto.edu/ivrii/monsterbook.pdf/},
pdflicenseurl={http://creativecommons.org/licenses/by-nc-nd/3.0/},
bookmarksdepth={4},
pdfpagelayout=SinglePage
}

\usepackage{eso-pic}
\usepackage[absolute,overlay]{textpos}

\definecolor{lightgray}{gray}{.92}
\numberwithin{equation}{chapter}

\newcommand{\sC}{\mathscr{C}}
\newcommand{\sL}{\mathscr{L}}

\newcommand{\cL}{\mathcal{L}}

\newcommand{\cW}{\mathcal{W}}
\newcommand{\cU}{\mathcal{U}}

\newcommand{\bZ}{\mathbb{Z}}
\newcommand{\bR}{\mathbb{R}}

\newcommand{\bS}{\mathbb{S}}

\newcommand{\W}{\mathsf{W}}
\newcommand{\T}{\mathsf{T}}
\newcommand{\w}{\mathsf{w}}

\newcommand{\N}{\mathsf{N}}

\newcommand{\supp}{\operatorname{supp}}

\newcommand{\rank}{\operatorname{rank}}

\newcommand{\x}{\mathsf{x}}
\newcommand{\y}{\mathsf{y}}

\newcommand{\cV}{\mathcal{V}}

\title{Pointwise Spectral Asymptotics  near Singularity \thanks{\emph{2010 Mathematics Subject Classification}: 35P20.}\thanks{\emph{Key words and phrases}: Microlocal Analysis, sharp  spectral asymptotics.}}

\author{Victor Ivrii\thanks{This research was supported in part by National Science and Engineering  Research Council (Canada) Discovery Grant  RGPIN 13827}}

\begin{document}

%

\maketitle

\begin{abstract}
 We establish semiclassical asymptotics and estimates for the $e_h(x,x;\tau)$ where
 $e_h(x,y,\tau)$ is the Schwartz kernel of the spectral projector for a second order elliptic operator inside domain with power singularity in the origin. While such asymptotics   for its trace $\N_h(\tau)= \int e_h(x,x,\tau)\,dx$ are well-known, the poinwise  asymptotics are much less explored.

Our main tools: microlocal methods, improved successive approximations and geometric optics methods.

\end{abstract}

\chapter{Introduction}
\label{sect-1}

In this paper we consider semiclassical diagonal poinwise asymptotics of $e_h(x,x,\tau)$ where  $e_h(x,y,\tau)$  is the Schartz kernel of the spectral projector for a \emph{toy-model operator}
\begin{gather}
A^0 _h = -\frac{1}{2} h^2\Delta - |x|^{-2\alpha} \qquad\text{with\ \ } 0<\alpha<1
\label{eqn-1.1}\\
\intertext{or a \emph{perturbed toy-model operator}}
 A_h = \sum_{j,k} g^{jk} (hD_j-V_j)(hD_k-V_k) +V,
 \label{eqn-1.2}
 \end{gather}
assuming that
\begin{align}
&D^\sigma ( g^{jk} - \updelta_{jk} )= O(|x|^{\beta-|\sigma|}),
\label{eqn-1.3}\\
&D^\sigma V_j(x)= O(|x|^{-\alpha +\beta-|\sigma|})
\label{eqn-1.4}\\
 \shortintertext{and}
 &D^\sigma(V(x)+|x|^{-2\alpha}) =O(|x|^{-2\alpha+\beta-|\sigma|})\qquad \forall \sigma\colon |\sigma|\le K 
 \label{eqn-1.5}
 \end{align}
 with $\beta>0$ and sufficiently large $K$.
 
 \section{Toy-model operator}
\label{sect-1.1}

Let $A_h\coloneqq A^0_h $ be toy-model operator (\ref{eqn-1.1}) in $\bR^d$. It is self-adjoint in $\sL^2(\bR^d))$. Let $E^0_h(\tau)$ be its spectral projector and $e^0_h(x,y,\tau)$ be a Schwartz kernel of $E^0_h(\tau)$.

\begin{theorem}\label{thm-1.1}
Consider toy-model operator $A^0_h$. Then 
 \begin{enumerate}[wide, label=(\roman*), labelindent=0pt]
\item\label{thm-1.1-i}
For $h\in (0,1]$, $|x|\le 1$, $|y|\le 1$, $T\ge  C_0$   the following estimate holds
\begin{gather}
|e _h(x,y, 0) -e^{\T}_{T,h}(x,y,0)|\le Ch^s T^{-s}
\label{eqn-1.6}\\
\intertext{with the \underline{Tauberian expression}}
e^\T_{T,h}(x,y,\tau) =h^{-1}  \int_{-\infty} ^\tau F_{t\to h^{-1}}\Bigl( \bar{\chi}_{T} u_h(x,y,\tau)\Bigr)\,d\tau  
\label{eqn-1.7}
\end{gather}
where here and below $u_h(x,y,t)$ is the Schwartz kernel of $e^{ith^{-1}A_h}$, \\ $\bar{\chi}\in \sC_0^([-1,1])$, $\chi ^-\in \sC_0([-1,-\frac{1}{2}])$, $\chi ^+\in \sC_0([\frac{1}{2},1])$, $\chi=\chi^++\chi^-$, $\chi_T(t)=\chi(t/T)$ etc. Here $K=K(s)$ in \textup{(\ref{eqn-1.3})}--\textup{(\ref{eqn-1.5})}.
\item\label{thm-1.1-ii}
For $x\colon |x|=1$
 \begin{gather}
|e^\T_{T,h}  (x,x,0) -e^\T_{T',h}(x,x,0)|\le 
C h^{\frac{3}{2}-d} + C\bar{\rho} ^{d-1}h^{1-d} 
\label{eqn-1.8} 
\shortintertext{with}
\bar{\rho}= \left\{\begin{aligned}
& h^{\frac{1}{2}-\delta} &&\text{as\ \ } (1-\alpha)^{-1}\notin 2\bZ, \\
& h^{\frac{1}{3}-\delta} &&\text{as\ \ } (1-\alpha)^{-1}\in 2\bZ.
\end{aligned}\right.
\label{eqn-1.9} 
\end{gather}
with an arbitrary small exponent $\delta>0$  and an arbitrarily small constant $T'>0$. 
\item\label{thm-1.1-iii}
Furthermore, as $0<\alpha \le \frac{1}{2}$ one can skip the first term in the right-hand expression of \textup{(\ref{eqn-1.8})}.\end{enumerate}
 \end{theorem}

\begin{remark}\label{rem-1.2}
We know that for small enough constant $T_*>0$
\begin{gather}
e^\T_{T_*,h}(x,x,\tau)=e^\W_{h}(x,x,\tau) + O(h^{2-d})
\label{eqn-1.10}\\
\shortintertext{with the \underline{Weyl expression}}
e^\W_{h}(x,x,\tau) = (2\pi h)^{-d} \varpi_d (2(\tau -  V(x)))^{\frac{d}{2}}
\label{eqn-1.11}
\end{gather}
with $\varpi_d$ a volume of $B(0,1)\subset\bR^d$.
\end{remark}

\section{General case}
\label{sect-1.2}

Consider  general operator (\ref{eqn-1.2}) in $\sL^2(X)$ with $0\in X\subset \bR^d$. 
If $\partial X\ne \emptyset$, assume that there are boundary conditions on $\partial X$.  

\begin{theorem}\label{thm-1.3}
Consider  general operator \textup{(\ref{eqn-1.2})} in $\sL^2(X)$ with $0\in X\subset \bR^d$. Assume that
conditions \textup{(\ref{eqn-1.3})}--\textup{(\ref{eqn-1.5})} are fulfilled in $B(0,2)$ and that
\begin{claim}
$A_h$ is a self-adjoint operator.
\label{eqn-1.12}
\end{claim}
Then

 \begin{enumerate}[wide, label=(\roman*), labelindent=0pt]
\item\label{thm-1.3-i}
For $h\in (0,1]$, $\bar{r}\le r=|x|\le r_0$, $\bar{r}\le r=|y|\le r_0$, $T\ge  C_0r_0^{1+\alpha}$ where $r_0$ is a small constant and $ \bar{r}=h^{1/ (1-\alpha)}$,  the following estimates hold
\begin{gather}
 |e _h(x,y, 0) -e^{\T}_{T,h}(x,y,0)|\le Ch^{1-d}r^{2\alpha-\alpha d}  
\label{eqn-1.13}
\end{gather}
and
\begin{multline}
|e _h(x,y, 0)-e^0_h(x,x,0) -e^{\T}_{T,h}(x,y,0)+ e^{0\,\T}_{T,h}(x,y,0)|\\
\le Ch^{1-d}r^{2\alpha-\alpha d+\beta}  
\label{eqn-1.14}
 \end{multline}
with the  Tauberian expression  $e^\T_{T,h}(x,y,\tau)$ defined by \textup{(\ref{eqn-1.7})}.
 \item\label{thm-1.3-ii}
 For $x\colon r_*= h^{1/(1-\alpha) -\delta}\le r=|x|\le r_0$
 \begin{multline}
|e^\T _{T,h}(x,x,0) -e^\T_{T',h}(x,x,0)|\\[3pt]
\le 
 C h^{\frac{3}{2}-d} r^{-\frac{3}{2}(1-\alpha) -\alpha d} +  C\bar{\rho} ^{d-1}h^{1-d} r^{-1+\alpha -\alpha d} 
 + C h^{1-d} r^{-1+\alpha -\alpha d +\beta d-\beta}
\label{eqn-1.15} 
\end{multline}
with 
\begin{gather}\bar{\rho}= \left\{\begin{aligned}
& \hbar^{\frac{1}{2}-\delta} &&\text{as\ \ } (1-\alpha)^{-1}\notin 2\bZ, \\
& \hbar^{\frac{1}{3}-\delta} &&\text{as\ \ } (1-\alpha)^{-1}\in 2\bZ,
\end{aligned}\right. \qquad \hbar= hr^{-1+\alpha}
\label{eqn-1.16} 
\end{gather}
with an arbitrary small exponent $\delta>0$, $K=K(s,\delta)$ in \textup{(\ref{eqn-1.3})}--\textup{(\ref{eqn-1.5})}  and 
$T'= \epsilon 'r^{1+\alpha}$ with an arbitrarily small constant $\epsilon '>0$.

 \item\label{thm-1.3-iii}
 Further, for $r_*\le r \le r^*=h^{1/(1-\alpha+\beta)}$
\begin{multline}
|e^\T _{T,h}(x,x,0) -e^\T_{T',h}(x,x,0)- e^{0\,\T} _{T,h}(x,x,0) + e^{0\,\T}_{T',h}(x,x,0)|\\[3pt]
\le 
 C h^{\frac{1}{2}-d} r^{-\frac{1}{2}(1-\alpha) -\alpha d +\beta} +  C\bar{\rho} ^{d-1}h^{-d} r^{\beta -\alpha d} 
 + C h^{-d} r^{ -\alpha d +\beta d}
\label{eqn-1.17} 
\end{multline}
 
\item\label{thm-1.3-iv}
Furthermore, as $0<\alpha \le \frac{1}{2}$ one can skip the first terms in the right-hand expressions of \textup{(\ref{eqn-1.15})} and \textup{(\ref{eqn-1.17})}.
\end{enumerate}
 \end{theorem}

\begin{remark}\label{rem-1.4}
 \begin{enumerate}[wide, label=(\roman*), labelindent=0pt]
\item\label{rem-1.4-i}
We know that in the framework of Theorem~\ref{thm-1.3}\ref{thm-1.3-ii}
\begin{gather}
|e^\T_{T',h}(x,x,0)- e^\W_{h}(x,x,0))|   
\le Ch^{2-d} r^{-2+2\alpha -\alpha d} + C h^{1-d} r^{-\alpha d+\alpha-1+\beta}.
\label{eqn-1.18}
\end{gather}
and in the framework of Theorem~\ref{thm-1.3}\ref{thm-1.3-ii}
\begin{multline}
|e^\T_{T',h}(x,x,0)- e^\W_{h}(x,x,0) - e^{0\,\T}_{T',h}(x,x,0) + e^{0\,\W}_{h}(x,x,0)|  \\[3pt]
\le Ch^{2-d} r^{-2+2\alpha +\beta-\alpha d}.
\label{eqn-1.19}
\end{multline}
\item\label{rem-1.4-ii}
Estimate (\ref{eqn-1.15}) without the last term is estimate (\ref{eqn-1.8}) scaled.

\item\label{rem-1.4-iii}
Combining estimates (\ref{eqn-1.13}), (\ref{eqn-1.15}) and (\ref{eqn-1.18}) we can estimate 
\begin{gather*}
e_h(x,x,0)-e^\W_h (x,x,0)
\end{gather*}
and combining estimates (\ref{eqn-1.13}),  (\ref{eqn-1.17})
and  (\ref{eqn-1.19}) we can estimate 
\begin{gather*}
e_h(x,x,0)-e^\W_h (x,x,0) - e^0_h(x,x,0) + e^{0\,\W}_h (x,x,0).
\end{gather*}
\end{enumerate}
\end{remark}

\section{Plan of the paper.}
\label{sect-1.3}
In Section~\ref{sect-2} we study a toy-model operator $A^0_h$. First, in Subsection~\ref{sect-2.1} we study corresponding Hamiltonian trajectories. In Subsection~\ref{sect-2.2} we study propagation of singularities, especially long-range propagation and propagation near $\{(x,\xi)\colon |x\times \xi|=0\}$, and derive microlocal and local pointwise spectral asymptotics with the improved remainder estimates. Finally, in Subsection~\ref{sect-2.3} we scale the previous results.

In Section~\ref{sect-3} we study operator $A_h$ which is perturbed a toy-model operator $A^0_h$. In Subsection~\ref{sect-3.1} we study propagation of singularities, especially long-range propagation and propagation near $\{(x,\xi)\colon |x\times \xi|=0\}$, and derive microlocal and local pointwise spectral asymptotics with the improved remainder estimates.  In Subsection~\ref{sect-3.2}  by the metod of successive approximations the difference $e_h(x,x,\tau)-e^0_h(x,x,\tau)$. Finally, in Subsection~\ref{sect-3.3} we consider cases
$|x|\le \bar{r}=h^{1/(1-\alpha)}$ and $\bar{r}\le |x|\le r_*=h^{1/(1-\alpha)-\delta}$.

In the next paper (or the next variant of this paper)  we consider out of diagonal asymptotics for such operators.

\chapter{Toy-model operator}
\label{sect-2}

\section{Hamiltonian trajectories}
\label{sect-2.1}

\paragraph{Geometry: $\mathbf{d=2}$.}
As $d=2$ trajectories on the energy level $0$ are defined in the polar coordinates by
\begin{gather}
\frac{d r}{dt} = \pm \sqrt{2 r^{-2\alpha} - M^2r^{-2}},
\label{eqn-2.1}\\
\frac{d\theta}{dt}=Mr^{-2}
\label{eqn-2.2}
\end{gather}
with \emph{angular momentum} $M$.  Then
\begin{gather}
d\theta = \pm \frac{Mr^{-2}\,dr}{\sqrt{2r^{-2\alpha}-M^2r^{-2}}}=
\pm \frac{M\rho^{-2}\, d\rho }{\omega \sqrt{2\rho^{-1} -M^2\rho^{-2}}},
\label{eqn-2.3}
\end{gather}
with $\rho =r^{\omega}$ with $\omega=2-2\alpha$, $0<\omega <2$.  The latter describes Kepler's parabolas as $M\ne 0$ in $(\theta,\rho)$ coordinates and therefore we need to perform a conformal transform $\rho \mapsto r=\rho^{1/\omega}$, $\theta\mapsto \omega^{-1}\theta$. Since increment of $\theta$ along parabola is $\pm 2\pi$, increment of $\theta$ along original Hamiltonian trajectory is $2\pi \omega^{-1}$ and therefore 
\begin{claim}\label{eqn-2.4}
Every Hamiltonian trajectory with $M\ne 0$ has
$n\coloneqq -1-\lceil -\omega ^{-1}\rceil =  -1-\lceil -(2-2\alpha )^{-1}\rceil$ self intersections. In particular, the number of self-intersections is $0$ for $0<\alpha\le \frac{1}{2}$,
$1$ for $\frac{1}{2}<  \alpha \le  \frac{3}{4}$, etc.
\end{claim}
All these trajectories are similar and obtained from one with $M=1$ by scaling and rotation. Because of this
\begin{claim}\label{eqn-2.5}
Let $\frac{1}{2}<\alpha <1$. Then through every point $x\ne 0$ there pass exactly $n$ Hamiltonian trajectories with $M=M^+_k(x)=\mu_k|x|^{2-2\alpha}$, $0<\mu_1<\ldots <\mu_n$, returning to
$x$ once after $k$ winding around $0$ in the counter-clockwise  direction, and also $n$ Hamiltonian trajectories with $M=M^-_k(x)=-\mu_k|x|^{1-\alpha}$, returning to
$x$ once after $k$ winding around $0$ in the clockwise  direction.
\end{claim}

\vskip-\baselineskip
\begin{claim}\label{eqn-2.6}
Let $\frac{1}{2}<\alpha <1$. Consider trajectory with $M=1$. Denote by $\gamma_k$, $k=n,n-1,\ldots, 1$ the radius of $(n-k+1)$-th intersection counting from ininity; then
$\gamma_n>\gamma_{n-1}>\ldots >\gamma_1$ and therefore $\mu_k = \gamma_k ^{\alpha-1}$ satisfy $\mu_1 >\ldots > \mu_n$.
\end{claim}

Let $\mu_0=0$.

\begin{claim}\label{eqn-2.7}
Let $x, y\in \bR^2\setminus 0$. There are $n$ trajectories from $x$ to $y$,  winding around $0$ in the counter-clockwise  direction either $1,2,\ldots , n$ or $0,1,\ldots , n-1$ times correspondingly, and $n$ trajectories from $x$ to $y$ windingaround $0$ in the clockwise  direction either $0,1,\ldots , n-1$ or $1,2,\ldots , n$ times correspondingly.
\end{claim}

Indeed,  consider two trajectories, issued from $x$ in the corresponding  direction, with the angular momenta $M_1$ and $M_2$, parametrized by $\theta$.
Then   $|r_1(\theta) -|x| | >  |r_2(\theta) -|x| |$ provided $\theta \gtrless 0$ and $|M_1|<|M_2|$.

\begin{figure}[h]
\centering
 \begin{tikzpicture}[xscale=4.1,yscale=4.1]
\draw (-1.5,-1) rectangle (1.5,1.);
 \clip (-1.5,-1) rectangle (1.5,1.);
\fill (0,0) circle (.015);
\fill[red] (.745,0) circle (.015);
\foreach \a in {240}
\draw[red, ultra thin, domain=180-\a : 1200-\a,scale=1.5,samples=500] plot ({\x+\a}: {( sin (\a/8)/ sin ((\x+\a) /8))^2});
\foreach \a in {325}
   \draw[blue, ultra thin, , domain=160-\a : 1270-\a,scale=1.5,samples=500] plot ({\x+\a+180}: {( sin (\a/8)/ sin ((\x+\a) /8))^2});
\foreach \a in {125}
   \draw[black, ultra thin, , domain=100-\a : 1320-\a,scale=1.5,samples=500] plot ({\x+\a+180}: {( sin (\a/8)/ sin ((\x+\a) /8))^2});
\end{tikzpicture}
\caption{\label{fig-1} This figure  shows several loops in the same (red) point, around origin (black point).}
\end{figure}
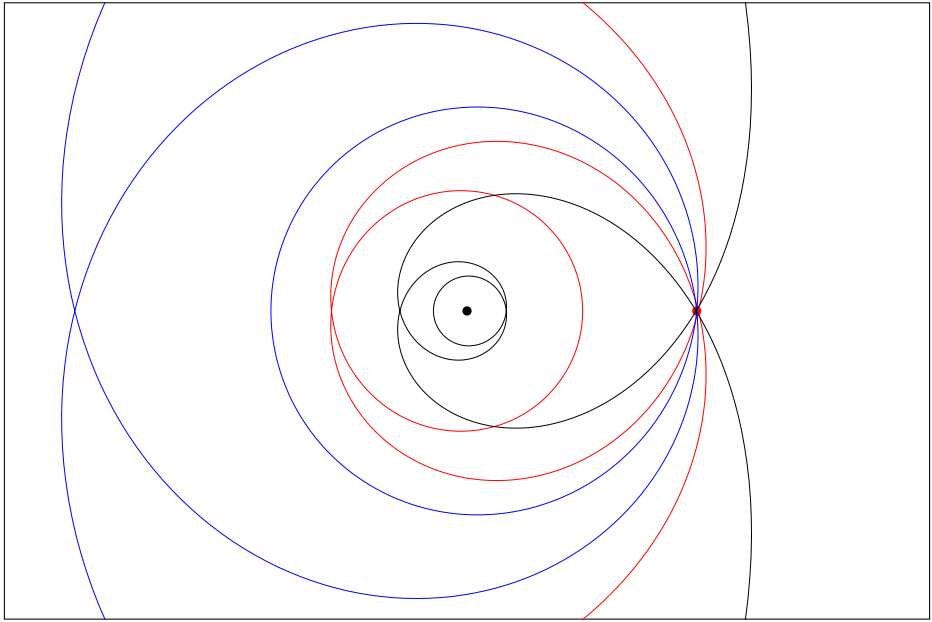

The trajectory  is (rotated)
\begin{gather}
r^\omega  = \frac{M^2}{(1-\cos ( \omega\theta))}
\label{eqn-2.8}.
\end{gather}
Indeed, it follows from the fact that  as $\omega=1$ the trajectory is a parabola $r= \frac{A}{1-\cos(\theta)}$ and therefore in the general case
$r^\omega = \frac{A}{(1-\cos ( \omega\theta))}$. Applying dynamic equations
\begin{gather}
\frac{dr}{dt}=\pm \sqrt{2 r^{\omega-2}- M^2 r^{-2}}, \label{eqn-2.9}\\
\frac{d\theta}{dt} =Mr^{-2}
\label{eqn-2.10}
\end{gather}
we conclude that $A^\omega=M^2$. 
Then self-intersections are  achieved as
\begin{gather}
2 \theta_k + 2\pi k =2\pi \omega^{-1}\implies \theta_k = \pi \omega^{-1}-\pi k,
\label{eqn-2.11}\\
\gamma_k ^\omega = \frac{1} {(1-\cos (\omega \pi k))},\qquad \mu_k =\gamma_k^{-\omega/2}=(1-\cos (\omega \pi k))^{\frac{1}{2}}.
\label{eqn-2.12}
\end{gather}

 \paragraph{Dynamics: $\mathbf{d=2}$.}
Consider trajectory passing through $x$ with $x$ being $k$-th points of self-intersection. Then $T_k(x)= t_k |x|^{1+\alpha}$. 
Meanwhile
\begin{multline}
dt=   r^{2} d\theta = M^{4/\omega}  (1-\cos ( \beta\theta))^{-2/\omega}\,d\theta\\ 
\begin{aligned} \implies t_k =& 2\mu_k ^{4/\omega} \int_{\theta_k}^{\pi }  (1-\cos ( \omega\theta))^{-2/\omega} \, d\theta\\ 
                             =& 2(1-\cos (\omega\theta_k))^{2/\omega} \int_{\theta_k}^{\pi }  (1-\cos ( \omega\theta))^{-2/\omega} \, d\theta .
                             \end{aligned}
\label{eqn-2.13}
\end{multline}

\begin{remark}\label{rem-2.1}
 \begin{enumerate}[wide, label=(\roman*), labelindent=0pt]
\item\label{rem-2.1-i}
We do not know how $t_k$ are ordered and even if they are ordered at all.
\item\label{rem-2.1-ii}
There is however the purely radial movement with $M=0$ which is the limit of trajectories as $M\to 0$. We denote its time by
\begin{gather}
t_0 = 2\int _0^1 r^\alpha \,dr = \frac{2}{1+\alpha} .
\label{eqn-2.14}
\end{gather}
\end{enumerate}
\end{remark}

\paragraph{$\mathbf{d\ge 3}$.} 
\begin{claim}\label{eqn-2.15}
As   $x=ky\ne 0$,  due to rotational symmetry (along straight line passing through $0$ and $x$) instead of $n$ separate curves we get $n$ $(d-2)$-parameter families.
\end{claim}
\vskip-\baselineskip

\begin{claim}\label{eqn-2.16}
On the other hand, if $x\neq  ky$ since the trajectory from $x$ to $y$ belongs to a single $2$-dimensional plane, passing through $x,y,0$ the number of trajectories is as for $d=2$.
\end{claim}

\section{Estimates}
\label{sect-2.2}

\begin{proposition}\label{prop-2.2}
Consider toy-model operator \textup{(\ref{eqn-1.1})}.  Then
\begin{enumerate}[wide, label=(\roman*), labelindent=0pt]
\item  \label{prop-2.2-i}
For $T\ge \epsilon $, $|y|\le C_0T^{1/(1+\alpha)}$ and  $|x|\ge 2C_0T^{1/(1+\alpha)}$ 
\begin{gather}
| F_{t\to h^{-1}\tau}  \bigl(\bar{\chi }_T(t)  u_h(x,y,t) \bigr)|  \le C  h^s T^{-s}   \qquad \forall\tau\colon |\tau|\le \epsilon_1 T^{-2\alpha/(1+\alpha)}.
\label{eqn-2.17}
\end{gather}
\item \label{prop-2.2-ii}
For $T\ge C_0$,  $|x|\le \epsilon T^{1/(1+\alpha)}$ and $|\y|\le \epsilon T^{1/(1+\alpha)}$ 
\begin{gather}
| F_{t\to h^{-1}\tau}  \bigl(\chi _T(t)  u_h(x,y,t) \bigr)|  \le C  h^s T^{-s}  \qquad \forall\tau\colon |\tau|\le \epsilon _1T^{-2\alpha/(1+\alpha)}.
 \label{eqn-2.18}
 \end{gather}
 \end{enumerate}
\end{proposition}

\begin{proof}
Let $\zeta \in \sC^\infty (\bR)$, $\zeta(t) =0$ as $t\ge 1$ and $\zeta (t)=1$ as $t\le 0$, $\zeta (t)<1$ as $0<t<1$.
\begin{enumerate}[wide, label=(\roman*), labelindent=0pt]
 \item\label{pf-2.2-i}
Assume first that  $|y|\le 1$, $|x|\ge 2$ and $T=\epsilon $. Then the proof is standard, based for $\pm t >0$ on $\phi (x,\xi ,t)\coloneqq \zeta (C_0t -|x|+ c)$ with an appropriate constant $c$; then
\begin{gather}
|\{\tfrac{1}{2}|\xi |^2 -|x|^{-2\alpha}, \,|x| \}| \le C_0\qquad \text{for\ \ } \  |\xi  |\le 1.
\label{eqn-2.19}
 \end{gather}
 One can see easily that in the commutator of $hD_t-A$ with $\phi$ all next terms acquire factors $h$ but not any negative powers of $x$.

After (\ref{eqn-2.17}) is proven in this special case we prove it in the general case by scaling $x\mapsto x R^{-1}$, $\xi  \mapsto \xi  R^{\alpha}$, $t\mapsto tR^{-1+\alpha}$, $h\mapsto hR^{-1+\alpha}$.

 \item\label{pf-2.2-ii}
Again consider first $T\asymp 1$, $|y|\le 1$. Let us use our standard method with $\phi (x,t,\xi )\coloneqq \zeta (-\epsilon _0 t +\langle x, \xi \rangle +c )$. Observe that
\begin{gather}
\{\tfrac{1}{2}|\xi |^2 -|x|^{-2\alpha}, \, \langle x,\xi \rangle\}= |\xi |^2 -2\alpha |x|^{-2\alpha}
\label{eqn-2.20}
\end{gather}
which on the energy level $0$ equals $2(1-\alpha) |x|^{-2\alpha}$. We select $c$ such that $\phi = 0$ as $|x|\le 1$,
$|\xi |^2 -2 |x|^{-2\alpha}=0$. Then according to \ref{prop-2.2-i}  we conclude that 
\begin{gather}
| F_{t\to h^{-1}\tau}  \bigl(\chi ^+ _T(t)  (I-Q^+) u_h(x,y,t) \bigr)|  \le C  h^s   \qquad \forall\tau\colon |\tau|\le \epsilon _2
\label{eqn-2.21}
\end{gather}
where $Q^+=q^{+\w}$, $\supp (q^+)\subset \{(x,\xi )\colon \langle x,\xi  \rangle \ge \epsilon \ge \epsilon_2 T^{(1-\alpha)/(1+\alpha)}\}$, which implies
\begin{gather*}
| F_{t\to h^{-1}\tau}  \bigl(\chi ^+ _T(t)   u_h(x,y,t) \bigr)|  \le C  h^s  \quad \forall x\colon |x|\le \epsilon_3 T^{1/(1+\alpha)}\  \forall\tau \colon |\tau|\le \epsilon _3.
\end{gather*}
Like in Part \ref{pf-2.2-i} one can see easily that in the commutator of $hD_t-A$ with $\phi $ all next terms acquire factors $h$ but not any negative powers of $x$.

Also, after  (\ref{eqn-2.18}) is proven in this special case we prove it in the general case by the same scaling  as  in Part \ref{pf-2.2-i}.
\end{enumerate}
\end{proof}

\begin{proof}[Proof of Theorem~\ref{thm-1.1}\ref{thm-1.1-i}]
It follows by standard Tauberian method from
\begin{gather}
|F_{t\to h^{-1} \tau} \bigl(\bar{\chi}_T u_h(x,y,\tau)\bigr) |\le Ch^{1-d}\qquad \forall \tau\colon |\tau|\le \epsilon
\label{eqn-2.22}
\end{gather}
with $x=y$ and small constant $T$ that
\begin{multline}
|e(x,x,\tau)-e(x,x,\tau')|\le Ch^{-d}|\tau-\tau'|+ Ch^{1-d}\\
\forall \tau,\tau'\colon |\tau|\le \epsilon_1, |\tau'|\le \epsilon _1
\label{eqn-2.23}
\end{multline}
which in turn implies (\ref{eqn-2.22}) with arbitrarily large $T$ and $C=C_T$; then we take $T=T_0$ and applying it
we conclude that (\ref{eqn-2.22}) holds for arbitrarily large $T$ with $C$ independent on $T$ and with arbitrary $x\colon |x|\le 1$, $y\colon |y|\le 1$, which in turn implies
(\ref{eqn-1.6}).
\end{proof}

Now we need to consider  $x,y\colon |x|\le 1,\, |y|\le 1$ and $T\le T_0$. 

 \begin{proposition}\label{prop-2.3}
Consider toy-model operator \textup{(\ref{eqn-1.1})}. Let   $Q_l=q_l^\w (x,hD)$ be operators such that
\begin{align}
&\supp(q_1)\subset \{ (x,\xi )\colon |x|\le 1,\ |x\times \xi |   \ge (1+\epsilon)\rho \}
\label{eqn-2.24}\\
 \shortintertext{and}
&\supp(q_2)\subset  \{(x,\xi )\colon |x|\le 1,\  \ |x\times \xi | \le (1-\epsilon)\rho  \}
 \label{eqn-2.25}
 \end{align}
where \underline{here and until further notice} $\rho>0$ is an arbitrarily small constant.  Then for $T\ge \epsilon_0$
\begin{gather}
|F_{t\to h^{-1}\tau}\bar{\chi}_T(t)Q_{2x}u_h (x,y,t) \,^t\! Q_{1y}  | \le C_T h^s \qquad \forall\tau\colon |\tau|\le \epsilon'.
\label{eqn-2.26}
\end{gather}
\end{proposition}

\begin{proof}
Easy proof by standard microlocal methods based on that operators $(x_kD_j -x_jD_k)$ commute with $A$, is left to the reader.
\end{proof}

 \begin{claim}\label{eqn-2.27}
 Note that in the framework of (\ref{eqn-2.24}) Hamiltonian trajectories passing through $\supp (q_1)$ are contained in $\{(x,\xi )\colon |x|\ge \epsilon ''\}$ and one can describe $u_h(x,y,t)\,^t\! Q_{1y} $ as semiclassical Fourier integral distribution  corresponding to the Lagrangian manifold 
$\Lambda = \{(x,\xi ,t; y, -\eta,\tau )\colon (y,\eta) =\Phi _t(x,\xi),\ \tau =a(x,\xi)\}$
where  $\Phi_t=\Phi_t(x,y) $ denotes a Hamiltonian flow generated by $a(x,\xi)$.
\end{claim}

\begin{proposition}\label{prop-2.4}
Consider toy-model operator \textup{(\ref{eqn-1.1})} with $d\ge2$. Let   $Q_1=q_1^\w(x,hD)$ be operator with the symbol satisfying \textup{(\ref{eqn-2.24})}.

\begin{enumerate}[wide, label=(\roman*), labelindent=0pt]
 \item\label{prop-2.4-i}
Then for $T\colon \epsilon_0\le T\le T^*$
\begin{gather}
|F_{t\to h^{-1}\tau} \chi_T(t) u_h (x,y,t) \,^t\! Q_{1y}  | \le C  h^{\frac{3}{2}-d} \qquad \forall\tau\colon |\tau|\le \epsilon'.
\label{eqn-2.28}
\end{gather}
 \item\label{prop-2.4-ii}
Furthermore, if $|\x\times y|\ge \rho$ then
 \begin{gather}
|F_{t\to h^{-1}\tau} \chi_T(t) u_h (x,y,t) \,^t\! Q_{1y}  | \le C h^{\frac{1}{2}(1-d) }  \qquad \forall\tau\colon |\tau|\le \epsilon'.
\label{eqn-2.29}
\end{gather}
\end{enumerate}
\end{proposition}

\begin{proof}
Let us fix $(\bar{x}, \bar{\xi})$ and assume that $\Phi _t (\bar{x},\bar{\xi })=(\bar{y},\bar{\eta})$  for $t=\bar{t}\ne 0$.
Consider in the vicinity of $(\bar{\xi},\bar{t})$ the set $\{\xi  \colon \Phi_t(\bar{x}, \xi )=\bar{y}\}$. We claim that
\begin{claim}\label{eqn-2.30}
If $\bar{x}\times \bar{y}=0$ then  $\rank (d_{\xi ,t } \uppi _x  \Phi_t(\bar{x}, \xi ))|_{\xi =\bar{\xi },t=\bar{t}} = 2$ 
\end{claim}
and
 \begin{claim}\label{eqn-2.31}
 If $\bar{x}\times \bar{y} \ne 0$ then $\rank (d_{\xi ,t} \uppi _x  \Phi_t(\bar{x}, \xi ))|_{\xi =\bar{\xi },t=\bar{t}}= d$.
 \end{claim}
 The easy proof of these claims  we leave to the reader. 

Then combined (\ref{eqn-2.27}) with (\ref{eqn-2.30}) and (\ref{eqn-2.31})  imply Statements \ref{prop-2.4-i} and \ref{prop-2.4-ii} respectively.
\end{proof}

\begin{remark}\label{rem-2.5}
 \begin{enumerate}[wide, label=(\roman*), labelindent=0pt]
 \item\label{rem-2.5-i}
Let $0<\alpha \le \frac{1}{2}$. Then the lack of self-intersections as $|x\times \xi |>0$ implies that if $Q_1=q_1^\w (x, hD)$ satisfies (\ref{eqn-2.24}) then for  $T\colon \epsilon_0\le T\le T^*$
\begin{gather}
|F_{t\to h^{-1}\tau} \chi_T(t) u_h (x,y,t) \,^t\! Q_{1y} \bigr|_{x=y} | \le C_T h^{s} \qquad \forall\tau\colon |\tau|\le \epsilon'.
\label{eqn-2.32}
\end{gather}
 \item\label{rem-2.5-ii}
Let $\frac{1}{2}<\alpha <1$. Then the lack self-intersections as $0< |x\times \xi | <(1+\epsilon)\rho $ with small enough constant $\rho >0$ implies that if $Q_1=q_1^\w (x, hD)$ satisfies 
\begin{gather}
\supp(q_1)\subset \{ (x,\xi )\colon |x|\le 1,\     (1-\epsilon)\rho' < |x\times \xi |  < (1+\epsilon)\rho \}
\label{eqn-2.33}
\end{gather}
with sufficiently small constant $\rho>0$ and arbitrarily small constant $\rho'\colon 0<\rho'<\rho$ then (\ref{eqn-2.32}) holds.
 \item\label{rem-2.5-iii}
Let $h^{\delta'} < \rho' <\rho \le \epsilon'$ with sufficiently small exponent $\delta'>0$. Then we still can apply  standard microlocal methods and prove Proposition~\ref{prop-2.3}. Furthermore, we can use  long-range propagation\footnote{\label{foot-1} See Subsection~\ref{monsterbook-sect-2-4-2-1},    \cite{monsterbook}.}\footnote{\label{foot-2} Because $D\Phi_t $ has a polynomial growth.}   and conclude that singularities  propagate along Hamiltonian trajectories. Then Statements \ref{rem-2.5-i} and \ref{rem-2.5-ii} of this Remark hold.
\end{enumerate}
\end{remark}

Now we need to consider the same expressions albeit with $q_1$ satisfying (\ref{eqn-2.33}) with $h^{1-\delta} <\rho' <\rho <h^{\delta'}$. We want to to prove first that the propagation is confined to the strip 
 \begin{gather*}
 \{(x,\xi)\colon  (1-\epsilon)\rho <  |x \times \xi| < (1+\epsilon)\rho \}
 \end{gather*} and  that there it happens along Hamiltonian trajectories.
 The most precise results (with the least restrictions on $\rho$) can be achieved in spherical coordinates $x=(\theta,r)\in \bS^{d-1}\times \bR^+$. Let $\xi=(\vartheta, \varrho)\in T_\theta^* \bS^{d-1}\times \bR$ be dual coordinates in the phase space. Then in $\sL^2(\bS^{d-1}\times \bR^+)$ operator is transformed to
 \begin{gather}
 A= h^2\Bigl(D_r^2 + \frac{1}{4} (d-1)(d-3)  r^{-2} + r^{-2} \Delta _\theta\Bigr) + V(r)
 \label{eqn-2.34}
 \end{gather}
 where $\Delta_\theta$ is a positive Laplacian on $\bS^{d-1}$. Then instead of Proposition~\ref{prop-2.3} we have
 
\begin{proposition}\label{prop-2.6}
Consider toy-model operator \textup{(\ref{eqn-1.1})}. Let   $Q_l=q_l^\w (\theta, hD_\theta)$ be operators such that
\begin{align}
&\supp(q_1)\subset \{ (x,\xi )\colon \ |\vartheta |   \ge (1+\epsilon)\rho \}
\label{eqn-2.35}\\
 \shortintertext{and}
&\supp(q_2)\subset  \{(x,\xi )\colon |x|\le 1,\  \ |\vartheta | \le (1-\epsilon)\rho  \}
 \label{eqn-2.36}
 \end{align}
where  $ h^{1-\delta}\le \rho \le h^{\delta'}$.  Then for $x=(\theta,r)$, $y=(\theta',r')$ and $T\ge \epsilon_0$
\begin{gather}
|F_{t\to h^{-1}\tau}\bar{\chi}_T(t)Q_{2x}u_h (\x,y,t) \,^t\! Q_{1y}  | \le C_T h^s \qquad \forall\tau\colon |\tau|\le \epsilon'.
\label{eqn-2.37}
\end{gather}
\end{proposition}

\begin{proof}
Proof is by the standard  elliptic microlocal arguments. 
\end{proof}

Next we want to prove that propragation is along the Hamiltonian trajectories. Since along such trajectory the minimum of $|x$ is $\asymp r(\rho)\coloneqq \rho ^{1/(1-\alpha)}$ we need to assume that $r(\rho)\gtrsim r_*=h^{1/(1-\alpha)-\delta}$ that is $\rho \ge h^{1-\delta}$
 with an arbitrarily small exponent $\delta>0$. This assumption we already made. 

As $r \asymp r(\rho)$ standard scaling  and long-range propagation results\footref{foot-1}  imply this if we take (after scaling) 
$\cW$ which is $\hbar^{\frac{1}{2}-\delta'}$-vicinity by $(y,\eta)$ singularities there for time $t\colon : |t|\le T= \hbar^{-\delta''}$ propagate to $\Phi_t (\cW)$. Here $\hbar= h\rho^{-1}$\,\footnote {\label{foot-3} Since scaling translates $x\to x^{-1}r^{-1}(\rho)$, $\xi \to \xi r^\alpha (\rho)$ and $h \to \hbar = h r^{-1+\alpha}(\rho)= h\rho^{-1}$.}.

Therefore we need to explore propagation from $(y,\eta)$ to $(x,\xi)$ and from  $(x,\xi)$ to $(y,\eta)$ assuming that $|y|\asymp 1$, $|\eta|\asymp 1$, $|x|\asymp r(\rho) \hbar^{-\delta'}$ and $|\xi|\asymp |x|^{-\alpha}$. 

Consider first corresponding Hamiltonian flow in the spherical coordinates. For fixed $\rho=|\vartheta|$ we have
\begin{multline}
\frac{dr}{dt}=   \varrho= \bigl(r^{-2\alpha}+2\tau -r^{-2}\rho^2\bigr)^{\frac{1}{2}}= 
 r^{-\alpha} \bigl(1+2r^{2\alpha}\tau -r^{2\alpha -2}\rho^2\bigr)^{\frac{1}{2}}\\[3pt]
\implies t = (1+\alpha)^{-1}r^{1+\alpha} + O\bigl(r^{3\alpha+1} |\tau|+ r^{3\alpha-1} \rho^2 \bigr)
\label{eqn-2.38}
\end{multline}
as long as 
\begin{gather}
C_0 r (\rho) \le r\le \epsilon_0 \tau_-^{-1/2\alpha};
\label{eqn-2.39}
\end{gather}
we count $t$ from the point with the minimal $|x|$. On the other hand,  we see that in the framework of (\ref{eqn-2.39})
\begin{gather}
\Phi _t (x,\xi) =  (r(t), \varrho(t); \Psi _s (\theta,\vartheta/\rho))
\label{eqn-2.40}
\end{gather}
where $\Psi_s$ is the geodesic flow on the $\bS^*(\bS^{d-1})$ and $s=s(t)$ 
\begin{multline}
ds = \rho r^{-2}dt= \rho r^{-2+\alpha}  \bigl(1+2r^{2\alpha}\tau -r^{2\alpha -2}\rho^2\bigr)^{-\frac{1}{2}} dr \\[3pt]
\implies
s =  s_0 -(1-\alpha)^{-1} \rho r^{-1+\alpha}  + O\bigl( r^{3\alpha-1}\rho  |\tau|+ r^{3\alpha-3} \rho^3\bigr) .
\label{eqn-2.41}
\end{multline}

\begin{remark}\label{rem-2.7}
According to the proof of long-term propagation theorem\,\footref{foot-2} 
\begin{gather}
e^{i h^{-1}tA_h}Q  = Q(t) e^{ih^{-1}tA_h} 
\label{eqn-2.42}
\end{gather}
provided $Q =q ^\w (x,hD)$ with the symbol satisfying 
\begin{gather}
\supp(q )\subset \{(x,\xi)\colon 1\le |x|\le 2,\ 1 \le |\xi|\le 2, \ \langle x,\xi\rangle >0\},
\label{eqn-2.43}\\[3pt]
|\partial_x^\mu \partial_\xi ^\nu q |\le C_{\mu\nu} h^{-\frac{1}{2} (1-\delta' ) (|\mu|+\nu|)},
\label{eqn-2.44}\\[3pt]
| D\Phi _t|\le h^{-\delta''} \quad\text{and} \quad T\le h^{-\delta''}
\label{eqn-2.45}
\end{gather}
with arbitrarily small exponent $\delta'>0$ ($\langle x,\xi\rangle >0$ means that the  propagation goes in the direction away from the origin). Here $Q( t)= q^\w (x,hD_x;t)$ is the solution of
\begin{gather}
Q_t (t) -ih^{-1}[A_h, Q(t)]=0
\label{eqn-2.46}\\
\intertext{and its symbol $q (t)$ satisfies}
|\partial_x^\mu \partial_\xi ^\nu q_2|\le C_{\mu\nu} h^{-\frac{1}{2}(1-\delta''' )(|\mu|+\nu|)}
\label{eqn-2.47}\\
\intertext{with $\delta''=\delta''(\delta')>0$,  $\delta'''=\delta'''(\delta')>0$ and}
\supp q (t) =\Phi_t (\supp (q )).
\label{eqn-2.48}
\end{gather}
\end{remark}

Applying  the  standard scaling we conclude that for symbol $q_n$ supported in 
$\{(x,\xi)\colon |x|\asymp r_n, |\xi| \asymp r_n^{-\alpha} \}$ with  $r_n\ge r_*$ we can  take  $T_n= r_n^{(1+\alpha)}h_n ^{-\delta''}$  with $h_n= hr_n^{-1+\alpha}$ and thus we arrive to a symbol $q'_{n+1}= q'(t_{n+1}-t_n)$ supported in
$\{(x,\xi)\colon |x|\asymp r_{n+1}, \ |\xi|\asymp r_{n+1}^{-\alpha}\}$ with $r_{n+1} = r_n h_n^{-\delta_4}$, so 
$h_{n+1} = h_n ^{1+\delta_5}$.

Based on (\ref{eqn-2.38})--(\ref{eqn-2.41}) we see that if $q$ satisfies in the spherical coordinates after corresponding rescaling (\ref{eqn-2.44}) then $q(t)$ satisfies in the corresponding scaling (\ref{eqn-2.44}) again, with \underline{the same exponent $\delta'$}.
Therefore with  $N=N(\delta, \delta_5)$ jumps we can go from $r_1\asymp r(\rho)$ to $r_N\asymp 1$. 

We want to do it in order to have
\begin{multline}
e^{ih^{-1}tA_h} Q_1 \equiv  e^{ih^{-1}tA_h} Q'_1Q_1 \equiv e^{ih^{-1}(t-T_1)A_h} Q_2 e^{ih^{-1}T_1}Q_1\\
\equiv  e^{ih^{-1}(t-T_1)A_h} Q'_2 Q_2 e^{ih^{-1}T_1}Q_1 \equiv 
e^{ih^{-1}(t-T_1-T_2)A_h} Q_3 e^{ih^{-1}T_2A_h} Q_2 e^{ih^{-1}T_1}Q_1\quad\  \\
\equiv 
Q'_{N+2}(t)e^{ih^{-1}(t-T_1-T_2-\ldots -T_N)A_h} Q_{N+1} e^{ih^{-1}T_NA_h}\cdots Q_3 e^{ih^{-1}T_3A_h} Q_2 e^{ih^{-1}T_1}Q_1
\label{eqn-2.49}
\end{multline} 
with $Q_n= q_n^\w (x,hD)$, $Q'_n= q_n^{\prime \w} (x,hD)$ with symbols satisfying after rescalings (\ref{eqn-2.43}) and (\ref{eqn-2.47}) respectively and $(I-Q'_n)Q_n \equiv 0$. 

Finally,  $|t-T_1-\ldots -T_N|\le T_{N+1}$. This  equality (\ref{eqn-2.49}) will imply the required propagation result, but we need to look at $\supp (q_n)$ and $\supp (q'_n)$. 

\begin{remark}\label{rem-2.8}
As mentioned, we do it in the spherical coordinates $(r,\theta)$. This will lead to the condition
\begin{gather}
\rho \ge h^{\frac{1}{2}-\delta}
\label{eqn-2.50}\\
\shortintertext{which is stronger than the original assumption} 
\rho \ge h^{1-\delta}.
\label{eqn-2.51}
\end{gather}
\end{remark}

\begin{proposition}\label{prop-2.9}
Consider toy-model operator \textup{(\ref{eqn-1.1})}.  Let $(\bar{x},\bar{\xi}) =\Psi_t (\bar{y}, \bar{\eta})$ and 
\begin{gather}
(1-\epsilon)\rho \le |\bar{x}\times \bar{\xi}|= |\bar{y}\times \bar{\eta}| =(1+\epsilon) \rho
\label{eqn-2.52} 
\end{gather}
with $\rho \ge h^{\frac{1}{2}-\delta}$. 
Let $Q_1=q_1^\w (x,hD)$ be operator with the symbol $q_1$ supported in $\hbar^{\frac{1}{2}-\delta'}(r_y,r_y^{-\alpha})$–vicinity of $(\bar{y},\bar{\eta})$ and $Q_2=q_2^\w (x,hD)$ be an operator with the symbol $q_2$ equal $1$ in 
$C_0\hbar^{\frac{1}{2}-\delta'}(r_x,r_x^{-\alpha})$–vicinity of $(\bar{x},\bar{\xi})$, where $r_x=|\bar{x}|$ and $r_y=|\bar{y}|$. 

Then
\begin{gather}
e^{ih^{-1}tA_h}Q_1 \equiv Q_2e^{ih^{-1}tA_h}Q_1. 
\label{eqn-2.53}
\end{gather}
\end{proposition}

\begin{proof}
 \begin{enumerate}[wide, label=(\roman*), labelindent=0pt]
 \item\label{pf-2.9-i}
Assume first that $|\uppi_x \Psi_{t'}|$ only increases as $t'$ changes from $0$ to $t$.  Consider  in the scaled spherical coordinates  
$\hbar^{\frac{1}{2}-\delta'}$-vicinity $\cU$ of $(\bar{y}, \bar{\eta})$. Then $\Psi_t(\cU)\subset \cV$ which is  
 $C_0\hbar^{\frac{1}{2}-\delta'}$-vicinity  of $\Phi_t (\bar{x},\bar{\xi})$ again in the scaled spherical coordinates.  
Combined with  (\ref{eqn-2.49}) implies (\ref{eqn-2.53}).  Here we can take any $\hbar\ge h^{1-\delta}r_y^{\alpha-1}$.

 \item\label{pf-2.9-ii}
 Similarly we can consider a case when $|\uppi_x \Psi_{t'}|$ only decreases as $t'$ changes from $0$ to $t$.
Here we can take any $\hbar\ge h^{1\delta}r_x^{\alpha-1}$.

 \item\label{pf-2.9-iii}
 Therefore such statement remains true if $|\uppi_x \Psi_{t'}|$ first decreases and then increases as $t'$ changes from $0$ to $t$.
 Here we must take $\hbar\ge h^{1-\delta}\rho^{-1}$. 
 \end{enumerate}
\end{proof}

\begin{proposition}\label{prop-2.10}
Consider toy-model operator \textup{(\ref{eqn-1.1})}.  Let  
\begin{gather}
|\bar{x}|=1,\qquad  a(\bar{x},\bar{\xi})=0, \qquad 0<\rho=|\bar{x}\times \bar{\xi}|\le \epsilon, 
\label{eqn-2.54} 
\end{gather}
with sufficiently small constant $\epsilon=\epsilon(\epsilon_0)>0$. Then
\begin{enumerate}[wide, label=(\roman*), labelindent=0pt]
 \item\label{prop-2.10-i}
If $(1-\alpha)^{-1}\notin 2\bZ$ then 
\begin{gather}
\min_{t\colon |t|\ge \epsilon_0} |\uppi_x \Psi_t (\bar{x},\bar{\xi})-\bar{x}|\asymp 1.
\label{eqn-2.55} 
\end{gather}

\item\label{prop-2.10-ii}
If $(1-\alpha)^{-1}\in 2\bZ$ then
\begin{gather}
\min_{t\colon |t|\ge \epsilon_0} |\uppi_x \Psi_t (\bar{x},\bar{\xi})-\bar{x}|\asymp \rho .
\label{eqn-2.56} 
\end{gather}
\end{enumerate}
\end{proposition}

\begin{proof}
\begin{enumerate}[wide, label=(\roman*), labelindent=0pt]
 \item\label{pf-2.10-i}
 If $(1-\alpha)^{-1}\notin 2\bZ$ then the rotation angle when $\rho\to 0$ tends to $\pi (1-\alpha)^{-1}\notin 2\pi \bZ$, which implies (\ref{eqn-2.55}). 
 
  \item\label{pf-2.10-ii}
 If $(1-\alpha)^{-1}=2n$, $n\in \bZ$ then the rotation angle when $\rho\to 0$ tends to $2\pi n$. Further, if $n=1$ then the trajectory is a parabola with the curvature radius  $\asymp  r_\rho=\rho^2$ near $0$  and then on the distance $\asymp 1$ from $0$
 the distance between branches is $\asymp \rho$. For $n\ge 2$ instead of this parabola we get its image under transformation $r\mapsto r^n$, $\theta = n\theta$ (if $d=2$) and then the distance between branches $\asymp r_\rho ^{1/2n}  = \rho$ again. 

Thus we arrive to (\ref{eqn-2.56}). 
 \end{enumerate}
\end{proof}

\begin{proposition}\label{prop-2.11}
Consider toy-model operator \textup{(\ref{eqn-1.1})}.  Let  $Q =q^\w (x,hD)$ be an operator with the symbol $q$ supported in 
$\epsilon ''\rho$-vicinity of $(\bar{x},\bar{\xi})$ satisfying \textup{(\ref{eqn-2.54})}  with sufficiently small constant $\epsilon>0$. Then
\begin{enumerate}[wide, label=(\roman*), labelindent=0pt]
 \item\label{prop-2.11-i}
If $(1-\alpha)^{-1}\notin 2\bZ$ and $\rho \ge h^{\frac{1}{2}-\delta}$ then
\begin{multline}
F_{t\to h^{-1}\tau}\Bigl(\chi_T(t) u_h (x,y,t) \,^t\!Q_{1y}\Bigr)\Bigr|_{y=x}=O(h^s)\\ \forall T\in [ \epsilon _1, T_1]  \ \forall \tau \colon |\tau|<\epsilon_2.
\label{eqn-2.57} 
\end{multline}
with an arbitrarily small constant $\epsilon_1>0$, arbitrarily large constant $T_1$ and $\epsilon_2=\epsilon_2(T_1)$.

\item\label{prop-2.11-ii}
If $(1-\alpha)^{-1}\in 2\bZ$ and 
\begin{gather}
\rho \ge h^{\frac{1}{3}-\delta}
\label{eqn-2.58} 
\end{gather}
 then \textup{(\ref{eqn-2.57})} holds.
\end{enumerate}
\end{proposition}

\begin{proof}
In virtue of Proposition~\ref{prop-2.9} we need to check that the left-side expression of (\ref{eqn-2.55}) is larger than 
$C\hbar^{\frac{1}{2}-\delta'}$. Recall that $\hbar= h/\rho$. 
\begin{enumerate}[wide, label=(\roman*), labelindent=0pt]
 \item\label{pf-2.11-i}
$(1-\alpha)^{-1}\notin 2\bZ$ we need to check $\hbar^{\frac{1}{2}-\delta'} \le \epsilon _3$ which follows from $\rho \ge h^{\frac{1}{2}-\delta}$ which we need for Proposition~\ref{prop-2.9}.

 \item\label{pf-2.11-ii}
$(1-\alpha)^{-1}\in 2\bZ$ we need to check $\hbar^{\frac{1}{2}-\delta'} \le \epsilon _3\rho$ which is equivalent to (\ref{eqn-2.58}).
 \end{enumerate}
\end{proof}

Now we need to check the contribution of zone $\{ (x,\xi)\colon |x\times \xi|<\rho\}$ with $\rho =h^{\frac{1}{2}-\delta}$ in the framework of Proposition~\ref{prop-2.11}\ref{prop-2.11-i} and $\rho =h^{\frac{1}{3}-\delta}$ in the framework of Proposition~\ref{prop-2.11}\ref{prop-2.11-ii}.

\begin{proposition}\label{prop-2.12}
Consider toy-model operator \textup{(\ref{eqn-1.1})}.  Let  $Q =q^\w (x,hD)$ be an operator with the symbol $q$ supported in 
$\epsilon ''\rho$-vicinity of $(\bar{x},\bar{\xi})$ satisfying 
\begin{gather}
|\bar{x}|=1,\qquad  a(\bar{x},\bar{\xi})=0, \qquad |\bar{x}\times \bar{\xi}|\le \rho, 
\label{eqn-2.59} 
\end{gather}
 Then
\begin{multline}
|F_{t\to h^{-1}\tau}\Bigl(\chi_T(t) u_h (x,y,t) \,^t\!Q_{1y}\Bigr)\Bigr|_{y=x}|\le C\rho^{d-1}h^{1-d} \\ \forall T\in [ \epsilon _1, T_1]  \ \forall \tau \colon |\tau|<\epsilon_2
\label{eqn-2.60} 
\end{multline}
with an arbitrarily small constant $\epsilon_1>0$, arbitrarily large constant $T_1$ and $\epsilon_2=\epsilon_2(T_1)$.
\end{proposition}

\begin{proof}
Proof follows from standard arguments and the standard estimate (\ref{eqn-2.60}) with $\chi_T(t)$ replaced by $\bar{\chi}_{T'}(t)$ with small enough constant $T'$.
\end{proof}

Taking the sum with respect to partition of unity we arrive to
 
 \begin{proposition}\label{prop-2.13}
 Consider toy-model operator \textup{(\ref{eqn-1.1})}.  Then 
  \begin{enumerate}[wide, label=(\roman*), labelindent=0pt]
  \item\label{prop-2.13-i} 
  For $|x|=1$
 \begin{multline}
|F_{t\to h^{-1}\tau}\Bigl(\chi_T(t) u_h (x,x,t) \Bigr)|\le
C h^{\frac{3}{2}-d} + C\bar{\rho} ^{d-1}h^{1-d} \\ 
\forall T\ge \epsilon_0\   \forall \tau \colon |\tau|<\epsilon_2 T^{-2\alpha/(1+\alpha)} 
\label{eqn-2.61} 
\end{multline}
 \item\label{prop-2.13-ii} 
Furthermore, as $0<\alpha \le \frac{1}{2}$ one can skip the first term in the right-hand expression of \textup{(\ref{eqn-2.61})}. 
\end{enumerate}\end{proposition}

\begin{proof}
Indeed, the first term in the  right-hand expression of \textup{(\ref{eqn-2.61})} is due to loops (which are absent as $0<\alpha \le \frac{1}{2}$ while the second term is a contribution of zone $\{\xi\colon |x\times \xi|\le \bar{\rho}\}$ with $\bar{\rho}$ defined by (\ref{eqn-1.9})..
\end{proof}

\begin{remark}\label{rem-2.14}
If we want to get rid off the first term in the right-hand expression of (\ref{eqn-2.61}) as $\frac{1}{2}< \alpha <1$, we need to calculate contribution of loops (in the zone $\{(x,\xi)\colon |x\times \xi|>\epsilon\}$. One can do it, arriving to the estimate
\begin{multline}
|F_{t\to h^{-1}\tau}\Bigl((\bar{\chi}_T(t) - \bar{\chi}_{T'}(t))u_h (x,x,t) \Bigr) \\
- \sum_{1\le k \le n} \sum_{\pm} \sum_{m\ge 0}\varkappa^\pm _{k,m} (x,\tau) \exp (\pm ih^{-1}\phi_k (x,\tau))h^{\frac{3}{2}-d+m}  |\le
  C\bar{\rho} ^{d-1}h^{1-d} 
  \label{eqn-2.62} 
\end{multline}
for all   $T\ge C_0$, $T'=\epsilon_0$ and  $ \tau \colon |\tau|<\epsilon_2 T^{-2\alpha/(1+\alpha)}$. 
\end{remark}

\begin{proof}[Proof of~Theorem~\ref{thm-1.1}]
Statements~\ref{thm-1.1-ii}  and~\ref{thm-1.1-iii} of Theorem~\ref{thm-1.1} follow from Statements~\ref{prop-2.13-i} and~\ref{prop-2.13-ii} of Proposition~\ref{prop-2.13}.
\end{proof}

\section{Scaling}
\label{sect-2.3}

Now we assume that
\begin{claim}\label{eqn-2.63}
Operator $A_h$ in question coincides with toy-model (\ref{eqn-1.1})  as $|x|\le R$, $R\ge 2$. 
\end{claim}
Then, assuming (\ref{eqn-2.12}) (that is, $A_h$ is a self-adjoint operator) we estimate  for $|x|\asymp |y|\asymp 1$
\begin{gather*}
|F_{t\to h^{-1}\tau}\bar{\chi}_T (t)u_h(x,y,t)|\le Ch^{1-d}\qquad \text{as\ \ }T\le \epsilon R^{1+\alpha}, \   |\tau|\le \epsilon R^{-2\alpha}.
\end{gather*}
 Taking $x=y$ we apply standard Tauberian arguments (Part I)\footnote{\label{foot-4} See \cite{monsterbook}, Chapter~4.}  to get estimate
\begin{multline*}
|e_h (x,x, \tau)- e_h (x,x, \tau')|\le C \bigl(h^{1-d}R^{-1-\alpha} + h^{-d}|\tau-\tau'|\bigr) \\
\text{as\ \ }  |\tau| +  |\tau'|\le \epsilon R^{-2\alpha}.
\end{multline*}

Now take $|x|=r\le 1$. Then scaling $x\mapsto r^{-1}|x|$, $\tau \mapsto \tau r^{2\alpha}$, $R\mapsto r^{-1}R$ and $h\mapsto \hbar\coloneqq hr^{\alpha-1}$ 
we arrive to
\begin{multline*}
|e_h (x,x, \tau)- e_h (x,x, \tau')| \le C r^{-(d-2)\alpha  }   \bigl(h^{1-d}R^{-1-\alpha}  + h^{-d} |\tau-\tau'|  \bigr)  \\[3pt]
\text{as\ \ }  |\tau| +  |\tau'|\le \epsilon  R^{-2\alpha}
\end{multline*}
which for $R=2$ becomes\begin{multline}
|e_h (x,x, \tau)- e_h (x,x, \tau')|\le C r^{-(d-2)\alpha  }  \bigl(h^{1-d} + h^{-d}|\tau-\tau'|  \bigr) \\[3pt]
\text{as\ \ }  |\tau| +  |\tau'|\le \epsilon
\label{eqn-2.64}
\end{multline}
provided
\begin{gather}
\bar{h}\coloneqq h^{1/(1-\alpha)}\le r \le 1
\label{eqn-2.65}
\end{gather}
where the last assumption means exactly  that $\hbar \le 1$. Then (\ref{eqn-2.64}) implies 
\begin{multline}
|e_h (x,y, \tau)- e_h (x,y, \tau')|\\[3pt]
\shoveright{\le C r_x^{-(d-2)\alpha /2 }   r_y^{-(d-2)\alpha /2 } \bigl(h^{1-d} + h^{-d}|\tau-\tau'|  \bigr)} \\[3pt]
\text{as\ \ }  |\tau| +  |\tau'|\le \epsilon, \ \  \bar{r} \le r _x=|x|\le 1 , \ \  \bar{r} \le r _y=|y|\le 1
\label{eqn-2.66}
\end{multline}
Apply standard Tauberian arguments (Part II)\footref{foot-4} we arrive to

\begin{proposition}\label{prop-2.15}
Consider self-adjoint operator coinciding with the toy-model operator \textup{(\ref{eqn-1.1})} as $|x|\le 2$.  Then 
\begin{multline}
|e_h(x,y,\tau) - e^\T_{T,h} (x,y,\tau) |\le C r_x^{-(d-2)\alpha /2 }  r_y^{-(d-2)\alpha /2 }h^{1-d}  \\[3pt]
\text{as\ \ }  |\tau| \le \epsilon, \ \  \bar{r} \le r _x=|x|\le 1 , \ \   \bar{r}  \le r _y=|y|\le 1, \ \  T= C_0(r_x+r_y)^{1+\alpha}
\label{eqn-2.67}
\end{multline}
with the \underline{Tauberian expression}
\begin{gather}
 e^\T_{T,h} (x,y,\tau ) =h^{-1} \int _{-\infty}^\tau \Bigl(F_{t\to h^{-1}\tau}\bar{\chi}_T (t)u_h(x,y,t)\Bigr) \,d\tau.
\label{eqn-2.68} 
\end{gather}
\end{proposition}

From now $x=y$. Then scaling results of Theorem~\ref{thm-1.1} and Remark~\ref{rem-1.2} and combining with (\ref{eqn-2.67})   we arrive to

\begin{proposition}\label{prop-2.16}
Consider self-adjoint operator coinciding with the toy-model operator \textup{(\ref{eqn-1.1})} as $|x|\le 2$.  Then for $\bar{r}\le r=|x|\le 1$
 \begin{enumerate}[wide, label=(\roman*), labelindent=0pt]
\item \label{prop-2.16-i}
Estimate holds 
\begin{multline}
|e _h(x,x,0) - e^\T_{T_*,h}(x,x,0)| \\
\le C r ^{-(d-2)\alpha  } h^{1-d} + 
 C r^{-\frac{3}{2}-\alpha (d-\frac{3}{2})} h^{\frac{3}{2}-d} + C\bar{\rho} ^{d-1} r^{-1-\alpha(d-1)}h^{1-d} 
\label{eqn-2.69}   
\end{multline}
with
\begin{gather}
\bar{\rho}= \left\{\begin{aligned}
& \hbar^{\frac{1}{2}-\delta} &&\text{as\ \ } (1-\alpha)^{-1}\notin 2\bZ, \\
& \hbar^{\frac{1}{3}-\delta} &&\text{as\ \ } (1-\alpha)^{-1}\in 2\bZ
\end{aligned}\right.
\label{eqn-2.70} 
\end{gather}
with $\hbar = hr^{\alpha-1}$ and  with  $T_*= \epsilon r^{1+\alpha}$. 

\item \label{prop-2.16-ii}
Further, as $0<\alpha \le \frac{1}{2}$ one can skip the second  term in the right-hand expression of \textup{(\ref{eqn-2.69})}.

\item \label{prop-2.16-iii}
Furthermore, 
\begin{gather}
e^\T_{T_*,h}(x,x,\tau)=e^\W_{h}(x,x,\tau) + O(r^{-2-(d-2)\alpha} h^{2-d})
\label{eqn-2.71}
\end{gather} 
with Weyl expression  $e^\W_{h}(x,x,0) $ defined by \textup{(\ref{eqn-1.11})}.
\end{enumerate}
\end{proposition}

\begin{remark}\label{rem-2.17}
Scaling, rotational symmetry and Remark~\ref{rem-2.14}  imply that for $\bar{r}\le r =|x|$ and $T_*=\epsilon_0 r^{1+\alpha}$
 \begin{gather}
e^\T _{T_*,h} \sim  \sum _{j=1}^{\infty} \kappa^0_j  (\tau |x|^{2\alpha}) h^{-d+2j} |x|^{-\alpha d -2j(1-\alpha)}.
\label{eqn-2.72}
\end{gather}
and for $T^*=C_0 r^{1+\alpha}$
\begin{multline} 
e^\T _{T,h} - e^\T _{T_*,h}\\[3pt]
 \sim   \sum _{1\le k\le n} \sum_{\pm} \sum_{j\ge 0}   
\kappa^{0\,\pm} _{k, j} (\tau |x|^{2\alpha}) h ^{\frac{3}{2}-d+j}   |x|^{-\alpha d - (\frac{3}{2}+j)(1-\alpha) } 
\exp\bigl( \pm ih^{-1}|x|^{1-\alpha}\phi^0_k(\tau |x|^{2\alpha} ) \bigr).
\label{eqn-2.73}
\end{multline}
\end{remark}

\chapter{General case}
 \label{sect-3}

\section{Standard results for perturbed operator}
 \label{sect-3.1}

Consider now operator $A_h$ defined by (\ref{eqn-1.2}) and satisfying 
(\ref{eqn-1.3})--(\ref{eqn-1.5}) in $B(0,R)\subset X$. Then in $B(0,R)$
 \begin{align}
&|D^\sigma ( g^{jk} - \updelta_{jk} )|\le c_\sigma \varepsilon |x|^{-|\sigma|},
\label{eqn-3.1}\\
&|D^\sigma V_j(x)|\le c_\sigma  \varepsilon |x|^{-\alpha-|\sigma|},
\label{eqn-3.2}\\
 &|D^\sigma(V(x)+|x|^{-2\alpha})| \le  c_\sigma  \varepsilon |x|^{-2\alpha-|\sigma|}\qquad \forall \sigma\colon |\sigma|\le K
\label{eqn-3.3}
 \end{align}
with $\varepsilon = R^\beta $.  We assume that
\begin{gather}
r_*\coloneqq h^{1/(1-\alpha) - \delta} \le r \le R
\label{eqn-3.4}
\end{gather}
with an arbitrarily small exponent $\delta>0$.

\begin{remark}\label{rem-3.1}
Usual scaling $x\mapsto xR^{-1}$, $\xi \mapsto \x R^{\alpha}$, $\tau \mapsto  \tau R^{2\alpha}$ and $h\mapsto \hbar = hR^{-1+\alpha}$ produces operator (\ref{eqn-1.2}) satisfying (\ref{eqn-3.1})--\ref{eqn-3.3}) with $\varepsilon =R^{\beta}$ in $B(0,R)$. In contrast to Subsection~\ref{sect-2.3} here $R\le \epsilon$. 

However below we   consider operator (\ref{eqn-1.2}) under assumptions (\ref{eqn-3.1})--(\ref{eqn-3.4}) for $R\ge r_*$ and then scale as necessary. This leads to different $\varepsilon$ in different applications.
\end{remark}

Repeating proof of Proposition~\ref{prop-2.2} we arrive to

\begin{proposition}\label{prop-3.2}
For operator \textup{(\ref{eqn-1.2})} under assumptions~\textup{(\ref{eqn-3.1})}--\textup{(\ref{eqn-3.3})}  in $B(0,R)$  both Statements \ref{prop-2.2-i} and \ref{prop-2.2-ii} of Propositions~\ref{prop-2.2} remain true as  $ r_*^{1+\alpha} \le T\le \epsilon_2 R^{1+\alpha}$, 
$\varepsilon \le \epsilon _0 $ where $\epsilon _0>0$, $\epsilon_2$  are  small enough constants.
\end{proposition}

Then similarly to Theorem~\ref{thm-1.1}\ref{thm-1.1-i}  and Proposition~\ref{prop-2.15} we arrive to

\begin{corollary}\label{cor-3.3}
Consider operator \textup{(\ref{eqn-1.2})} under assumptions~\textup{(\ref{eqn-3.1})}--\textup{(\ref{eqn-3.3})}  in $B(0,2)$. Then  \textup{(\ref{eqn-2.67})} remains true.
\end{corollary}

\begin{proof}[Proof of Theorem~\ref{thm-1.3}\ref{thm-1.3-i}]
We take $R=2r_0$ so that $\varepsilon \le \epsilon_0$.
\end{proof}

Similarly to  Proposition~\ref{prop-2.3} we have

\begin{proposition}\label{prop-3.4}
Consider operator \textup{(\ref{eqn-1.2})} under assumptions~\textup{(\ref{eqn-3.1})}--\textup{(\ref{eqn-3.3})}  in $B(0,2)$. Let
assumptions   \textup{(\ref{eqn-2.24})}--\textup{(\ref{eqn-2.25})} be fulfilled with 
\begin{gather}
\rho \ge C \varepsilon 
\label{eqn-3.5}
\end{gather}
where here $\rho$ is a small constant and $C=C(\epsilon)$ with $\epsilon$ from  \textup{(\ref{eqn-2.24})}--\textup{(\ref{eqn-2.25})}. Then  \textup{(\ref{eqn-2.26})} remains true for   $|x|\le 1$, $|y|\le 1$, $\epsilon_1\le T\le c_0$.
\end{proposition}

\begin{proof}
In contrast to the toy-model operator when $\{a(x,\xi) , x_j \xi_k -x_k\xi_j\}=0$ we have now that 
on the energy level $\tau\colon |\tau|\le \epsilon_0 |x|^{-2\alpha}$
\begin{gather}
|\{a(x,\xi) , x_j \xi_k -x_k\xi_j\}| \le C_0\varepsilon |x|^{-2\alpha}.
\label{eqn-3.6}
\end{gather}
Then we can apply standard propagation results. More sophisticated arguments will be required when both
$\rho$ and $\varepsilon$ are small parameters.
\end{proof}

Repeating proof of Proposition~\ref{prop-2.4} we arrive to

\begin{proposition}\label{prop-3.5}
Consider operator \textup{(\ref{eqn-1.2})} under assumptions~\textup{(\ref{eqn-3.1})}--\textup{(\ref{eqn-3.3})}  in $B(0,R)$. Let   $Q_1=q_1^\w(x,hD)$ be operator with the symbol satisfying \textup{(\ref{eqn-2.24})} with an arbitrarily small constant $\rho$.

Then both Statements~\ref{prop-2.4-i} and~\ref{prop-2.4-ii}  of Proposition~\ref{prop-2.4} remain true for $T\ge \epsilon_1$ provided 
$\varepsilon $ satisfies \textup{(\ref{eqn-3.5})}.
\end{proposition}

\begin{remark}\label{rem-3.6}
 \begin{enumerate}[wide, label=(\roman*), labelindent=0pt]
\item\label{rem-3.6-i}
Under assumption (\ref{eqn-3.5})  Statements~\ref{rem-2.5-i}--\ref{rem-2.5-ii} of Remark~\ref{rem-2.5} remain true. Furthermore,  Statement~\ref{rem-2.5-iii} remains true due to long-range propagation of singularities\,\footref{foot-1} and Proposition~\ref{prop-3.7} below.

\item\label{rem-3.6-ii}
For operator (\ref{eqn-1.2})  instead of  (\ref{eqn-2.30}) and (\ref{eqn-2.31}) we have respectively
\begin{claim}\label{eqn-3.7}
If $|\bar{x}\times \bar{y}|\le C_0\varepsilon $ then $\rank (d_{\xi ,t } \uppi _x  \Phi_t(\bar{x}, \xi ))|_{\xi =\bar{\xi },t=\bar{t}} \ge 2$ 
\end{claim}
\vskip-.5\baselineskip
and
\vskip-.5\baselineskip
\begin{claim}\label{eqn-3.8}
If $|\bar{x}\times \bar{y}|\ge C_0\varepsilon $ then $\rank (d_{\xi ,t } \uppi _x  \Phi_t(\bar{x}, \xi ))|_{\xi =\bar{\xi },t=\bar{t}} = d$. 
\end{claim}
\item\label{rem-3.6-iii}
I believe that  for generic perturbations   ${\rank (d_{\xi ,t } \uppi _x  \Phi_t(\bar{x}, \xi ))|_{\xi =\bar{\xi },t=\bar{t}} = d}$ even as
$|\bar{x} \times \bar{y}|\le C_0\varepsilon$;  morover, the largest absolute value of $d\times d$–minors of $d_{\xi ,t } \uppi _x  \Phi_t(\bar{x}, \xi ))$ is
$\asymp \varepsilon^{d-2}$.
\end{enumerate}
\end{remark}

\begin{proposition}\label{prop-3.7}
Consider operator \textup{(\ref{eqn-1.2})} under assumptions~\textup{(\ref{eqn-3.1})}--\textup{(\ref{eqn-3.3})} in $B(0,2)$. Then in the zone $\{(x,\xi)\colon |x\times \xi |\ge \rho\}$ with $\rho \ge C\varepsilon $ Hamiltonian flow $\Phi _t $ satisfies
\begin{gather}
|D_{x,\xi}  \Phi _t (x,\xi)| \le C \Bigl( 1 +  \frac{|t|}{r_x ^{1+\alpha}}\Bigr)^N.
\label{eqn-3.9}
\end{gather}
\end{proposition}

\begin{proof}
It follows easily from 
\begin{gather}
\frac{d\ }{dt} x|x|^{-1} = \sqrt{2} |x|^{-1-\alpha} \Bigl( 1+ O(\rho +\varepsilon)\Bigr).
\label{eqn-3.10}
\end{gather}
\end{proof}

Combining with long-term propagation of singularities\footref{foot-1} we arrive to

\begin{proposition}\label{prop-3.8}
Consider operator \textup{(\ref{eqn-1.2})} under assumptions~\textup{(\ref{eqn-3.1})}--\textup{(\ref{eqn-3.3})} in $B(0,2)$. Then  \begin{enumerate}[wide, label=(\roman*), labelindent=0pt]
\item\label{prop-3.8-i}
Statement of Proposition~\ref{prop-2.6} remains true in the zone $\{(x,\xi)\colon |x\times \xi |\ge \rho\}$ with 
\begin{gather}
\rho \ge h^{1-\delta} + C\varepsilon. 
\label{eqn-3.11}
\end{gather}

\item\label{prop-3.8-ii}
Statement of Remark~\ref{rem-2.7} remains true in the zone $\{(x,\xi)\colon |x\times \xi |\ge \rho\}$ with 
\begin{gather}
\rho \ge h^{\frac{1}{2}-\delta} + C\varepsilon.
\label{eqn-3.12}
\end{gather}
\end{enumerate}
\end{proposition}

Continuing, we arrive to

\begin{proposition}\label{prop-3.9}
Consider operator \textup{(\ref{eqn-1.2})} under assumptions~\textup{(\ref{eqn-3.1})}--\textup{(\ref{eqn-3.3})} in $B(0,2)$.
Let condition \textup{(\ref{eqn-3.11})} be fulfilled. Then
\begin{enumerate}[wide, label=(\roman*), labelindent=0pt]
\item\label{prop-3.9-i}
Statements of Remark~\ref{rem-2.8} and Proposition~\ref{prop-2.9} remain true.

\item\label{prop-3.9-ii}
Both Statements \ref{prop-2.10-i} and  \ref{prop-2.10-ii}  of Proposition~\ref{prop-2.10} remain true.

\item\label{prop-3.9-iii}
Both Statements \ref{prop-2.11-i} and  \ref{prop-2.11-ii}  of Proposition~\ref{prop-2.11} remain true.

\item\label{prop-3.9-iv}
Statement of Proposition~\ref{prop-2.12} remains true. 
\end{enumerate}
\end{proposition}

 Taking the sum with respect to partition of unity (like in Proposition~\ref{prop-2.13}) we see that  estimate (\ref{eqn-2.61}) holds with $\bar{\rho}$ replaced by $\bar{\rho}_\varepsilon \coloneqq  \bar{\rho} +C \varepsilon$:
 
 \begin{proposition}\label{prop-3.10}
Consider operator \textup{(\ref{eqn-1.2})}.  Assume that  \textup{(\ref{eqn-3.1})}--\textup{(\ref{eqn-3.3})} with small parameter $\varepsilon$ are fulfilled in $B(0,2)$. 
Also assume that the same assumptions albeit with $\varepsilon $ replaced by $\epsilon_2$ which is a small constant are fulfilled in $B(0,R)$ with $R\ge 2$.  Then for  $|x|=1$, 
 \begin{multline}
|F_{t\to h^{-1}\tau}\Bigl(\chi_T(t) u_h (x,x,t) \Bigr)|\le
C h^{\frac{3}{2}-d} + C\bar{\rho} ^{d-1}h^{1-d} + C\varepsilon ^{d-1}h^{1-d}   \\ 
 \forall T\ge \epsilon_0\   \forall \tau \colon |\tau|<\epsilon_2 T^{-2\alpha/(1+\alpha)} 
\label{eqn-3.13} 
\end{multline}
for  $\epsilon_1 \le T\le T^*\coloneqq \epsilon_1R^{1/(1+\alpha)}$  with $\bar{\rho}$ defined by \textup{(\ref{eqn-1.9})}.

Furthermore, as $0<\alpha \le \frac{1}{2}$ one can skip the first term in the right-hand expression of \textup{(\ref{eqn-3.13})}.
 \end{proposition}
  
\begin{remark}\label{rem-3.11}
Here we need assumptions   (\ref{eqn-3.1})--(\ref{eqn-3.3}) twice: 
\begin{enumerate}[wide, label=(\roman*), labelindent=0pt]
\item\label{rem-3.11-i}
In $B(0,2)$ with small parameter $\varepsilon$ to estimate the 
contribution to the left-hand expression of (\ref{eqn-2.60}) with $\chi_T$ replaced by $\chi^\pm _T$ of zone 
\begin{gather*}
\{\xi\colon |x\times \xi|\le \bar{\rho},\ \pm \langle x,\xi\rangle < 0\}
\end{gather*}
 where singularities propagate first towards $0$, 
\item\label{rem-3.11-ii}
In $B(0,R)$ with the small constant $\epsilon_2$ and $R\gg 1$ to ensure that the singularities propagating away from $0$ will continue this until time $T^*$. 

\item\label{rem-3.11-iii}
If the my conjecture (see Remark~\ref{rem-3.6}\ref{rem-3.6-iii}) is correct, then for generic perturbations with $\varepsilon\ge h$
one can replace the first term  in the right-hand expression of (\ref{eqn-3.13}) by $Ch^{(1-d)/2}  \varepsilon^{1-d/2}$. This improvement will be carried out to estimates (\ref{eqn-3.14}) and (\ref{eqn-1.15}).
\end{enumerate}
\end{remark}

Standard Tauberian arguments imply

\begin{corollary}\label{cor-3.12}
Consider  operator \textup{(\ref{eqn-1.2})}.  Let assumptions of Proposition~\ref{prop-3.10} be fulfilled. 
Then for $x\colon |x|=1$
 \begin{multline}
|e _h(x,x,0) -e^\T_{T_*,h}(x,x,0)| \\ 
\le CT^{*\,-1} h^{1-d}+ C h^{\frac{3}{2}-d} + C\bar{\rho} ^{d-1}h^{1-d} +  C\varepsilon ^{d-1}h^{1-d}
\label{eqn-3.14} 
\end{multline}
with the Tauberian expression \textup{(\ref{eqn-2.66})}, $T^*= \epsilon_1R^{1/(1+\alpha)}$
and arbitrarily small constant $T_*>0$. 

Furthermore, as $0<\alpha \le \frac{1}{2}$ one can skip the second  term in the right-hand expression of \textup{(\ref{eqn-3.14})}.
\end{corollary}

\begin{proof}[Proof of Theorem~\ref{thm-1.3}\ref{thm-1.3-ii}, \ref{thm-1.3-iv}]
Now standard scaling and setting $\varepsilon = C_0r^\beta$ imply  estimate (\ref{eqn-1.15})
(that is Theorem~\ref{thm-1.3}\ref{thm-1.3-ii}).

Furthermore, we see that as  $0<\alpha \le \frac{1}{2}$ one can skip the first term in the right-hand expression of this estimate (that is, part of Theorem~\ref{thm-1.3}\ref{thm-1.3-iv}).
\end{proof}

\begin{remark}\label{rem-3.13}
We know that for $|x|\asymp 1$ and small enough constant $T_*>0$  and $|x|=1$
\begin{gather}
e^\T_{T_*,h}(x,x,\tau)=e^\W_{h}(x,x,\tau) + O(h^{2-d}+\varepsilon h^{1-d})
\label{eqn-3.15} 
\intertext{with the Weyl expression}
e^\W_{h}(x,x,0) = (2\pi h)^{-d} \varpi_d g(x)^{-\frac{1}{2}}(2(\tau - V(x))^{\frac{d}{2}}, \qquad g =\det (g^{jk}).
\label{eqn-3.16} 
\end{gather}
\end{remark}

Now standard scaling and setting $\varepsilon = C_0r^\beta$ imply  estimate (\ref{eqn-1.18}) 
(that is Remark~\ref{rem-1.4}\ref{rem-1.4-i}).

\section{Perturbations of $e_h(x,y,0)$}
 \label{sect-3.2}

Let us denote unperturbed operator as $A^0_h$ and corresponding spectral function and propagator as $e^0_h(x,y,\tau)$ and $u^0_h(x,y,t)$. 

Consider in the framework of Proposition~\ref{prop-3.10} for $|x|\asymp |y|\asymp 1$
 \begin{multline}
 F_{t\to h^{-1}\tau} \bigl(\bar{\chi}_T(t) u_h(x,y,t)\bigr)=
F_{t\to h^{-1}\tau} \bigl(\bar{\chi}_{T'}(t) u_h(x,y,t)\bigr)\\ 
+  F_{t\to h^{-1}\tau} \bigl((\bar{\chi}_T(t) - \bar{\chi}_{T'}(t))u_h(x,y,t)\bigr)
\label{eqn-3.17}
 \end{multline}
with $T=T^*=C_0$ and $T'=T_*=\epsilon_0$ and corresponding formula for $A^0_h$ and $u^0_h(x,y,t)$. 

Let us integrate both by $\tau$ and multiply by $h^{-1}$.  Then taking the difference we arrive to
 \begin{multline}
 e^\T_{T,h} (x,y,\tau) - e^{0,\T}_{T,h} (x,y,\tau) =
  e^\T_{T',h} (x,y,\tau) - e^{0,\T}_{T',h} (x,y,\tau) \\
 +  h^{-1}\int_{-\infty}^\tau   F_{t\to h^{-1}\tau' } \Bigl((\bar{\chi}_T(t) - \bar{\chi}_{T'}(t))\bigl( u_h(x,y,t)-  u^0_h(x,y,t)\bigr)\Bigr)\,d\tau '.
  \label{eqn-3.18}
  \end{multline}

 \begin{proposition}\label{prop-3.14}
Consider operator \textup{(\ref{eqn-1.2})}.  Assume that  \textup{(\ref{eqn-3.1})}--\textup{(\ref{eqn-3.3})} are fulfilled in $B(0,2)$. 
Then for $|x|=1$ and arbitrarily small constant $T'$
 \begin{multline}
| e^\T_{T',h} (x,y,\tau) - e^{0,\T}_{T',h} (x,y,\tau) -  e^\W_{h} (x,y,\tau) + e^{0,\W}_{h} (x,y,\tau)| \\[3pt]
\le C\varepsilon h^{2-d}+ C'h^s\qquad \forall  \tau \colon |\tau|\le \epsilon.
\label{eqn-3.19}
\end{multline}
\end{proposition}

\begin{proof}
Let $A'_h=\varepsilon ^{-1}(A_h-A^0_h)$. Then $A_h\coloneqq  A_{\varepsilon,h} = A^0_h +\varepsilon A'_h$ and ecomposition
\begin{gather}
 e^\T_{T',h} (x,y,\tau)\sim \sum_{m\ge 0} \varkappa_m(x,\varepsilon, \tau)h^{-d+2m}
 \label{eqn-3.20}
 \end{gather}
 and similar decomposition for  $e^{0\,\T}_{T',h} (x,y,\tau)$ together with
 \begin{gather}
 \varkappa_m(x,\varepsilon, \tau) \sim \sum_{j\ge 0} \kappa_{m,j}(x,\tau)\varepsilon ^j
 \label{eqn-3.21}
 \end{gather}
 imply (\ref{eqn-3.19}). 
\end{proof}

\begin{proposition}\label{prop-3.15}
Consider operator \textup{(\ref{eqn-1.2})}.  Assume that  \textup{(\ref{eqn-3.1})}--\textup{(\ref{eqn-3.3})} with small parameter 
$\varepsilon \le h$ are fulfilled in $B(0,2)$. 
Also assume that the same assumptions albeit with $\varepsilon $ replaced by $\epsilon_2$ which is a small constant are fulfilled in $B(0,R)$ with $R$.  Then for  $|x|=1$, 
 \begin{multline}
 |h^{-1}\int_{-\infty}^\tau   F_{t\to h^{-1}\tau' } \Bigl((\bar{\chi}_T(t) - \bar{\chi}_{T'}(t))\bigl( u_h(x,y,t)-  u^0_h(x,y,t)\bigr)\bigr|_{y=x}\Bigr)\,d\tau'|\\
 \le C\varepsilon h^{\frac{1}{2}-d} + C\varepsilon \bar{\rho}^{d-1}h^{-d} + C'h^s
  \label{eqn-3.22}
 \end{multline}
 for  $\epsilon_1 \le T\le T^*\coloneqq \epsilon_1R^{1/(1+\alpha)}$  with $\bar{\rho}$ defined by \textup{(\ref{eqn-1.9})}.
 
Furthermore, as $0<\alpha \le \frac{1}{2}$ one can skip the first term in the right-hand expression of \textup{(\ref{eqn-3.22})}.
\end{proposition}

\begin{proof}
\begin{enumerate}[wide, label=(\roman*), labelindent=0pt]
\item\label{pf-3.15-i}
Let $Q=q^\w(x,hD)$ be an operator with the symbol supported in $\{x\colon  1-\epsilon_3 \le |x|\le 1+\epsilon_3,\ |x\times \xi|\ge \epsilon_3\}$. Then 
\begin{multline}
h^{-1}\int_{-\infty}^\tau   F_{t\to h^{-1}\tau' } \Bigl((\bar{\chi}_T(t) - \bar{\chi}_{T'}(t))  u_h(x,y,t)\,^t\!Q_y\bigr|_{y=x}\Bigr)\,d\tau'\sim\\
\sum _k  \sum_{m\ge 0}  e^{ih^{-1}\varphi _k(x, \varepsilon,\tau)} \varkappa_{k,m} (x, \varepsilon,\tau) h^{\frac{3}{2}-d+m} 
  \label{eqn-3.23}
\end{multline}
and the same is true for $u^0_h(x,y,t)$. Moreover 
\begin{gather}
\varphi _k(x, \varepsilon,\tau)\sim \sum_{j\ge 0} \phi_{k,j} (x,\tau)\varepsilon^j, \qquad
\varkappa _{k,m}(x, \varepsilon,\tau)\sim \sum_{j\ge 0} \kappa_{k,m,j} (x,\tau)\varepsilon^j
  \label{eqn-3.24}
\end{gather}
 and therefore the difference between  expression   (\ref{eqn-3.23}) and (\ref{eqn-3.23}) with $\varepsilon=0$ does not exceed $C\varepsilon h^{\frac{1}{2}-d}+C'h^s$. 
 
 \item\label{pf-3.15-ii}
Let $Q=q^\w(x,hD)$ be an operator with the symbol supported in 
$\{x\colon 1-\epsilon_3 \le |x|\le 1+\epsilon_3,\ \bar{\rho}\le |x\times \xi|\le \epsilon_3\}$. Then the same expressions are $O(h^s)$.

\item\label{pf-3.15-iii}
Finally, let $Q=q^\w(x,hD)$ be an operator with the symbol supported in 
$\{x\colon 1-\epsilon_3 \le |x|\le 1+\epsilon_3 ,\  |x\times \xi|\le \bar{\rho}\}$. Then  from the standard perturbation method\footnote{\label{foot-5} See \cite{monsterbook}.} follows that 
\begin{gather*}
 |h^{-1}\int_{-\infty}^\tau   F_{t\to h^{-1}\tau' } \Bigl((\bar{\chi}_T(t) - \bar{\chi}_{T'}(t))\bigl( u_h(x,y,t)-  u^0_h(x,y,t)\bigr)\,^t\!Q_y\bigr|_{y=x}\Bigr)\,d\tau'|
 \end{gather*}
 does not exceed $C\varepsilon h^{-1} \times  \bar{\rho}^{d-1}h^{1-d}$ which is the second term in the right-hand expression of (\ref{eqn-3.22}).
\end{enumerate}
\end{proof}

\begin{corollary}\label{cor-3.16}
Let $\varepsilon \le h$. Then 
\begin{multline}
 |h^{-1}\int_{-\infty}^\tau   F_{t\to h^{-1}\tau' } \Bigl((\bar{\chi}_T(t) - \bar{\chi}_{T'}(t))\bigl( u_h(x,y,t)-  u^0_h(x,y,t)\bigr)\bigr|_{y=x}\Bigr)\,d\tau'|\le \\
 C\varepsilon h^{2-d} + C\varepsilon h^{\frac{1}{2}-d} + C\varepsilon \bar{\rho}^{d-1}h^{-d}+C'h^s.
 \label{eqn-3.25}
  \end{multline}
 Furthermore, as $0<\alpha \le \frac{1}{2}$ one can skip the second term in the right-hand expression of \textup{(\ref{eqn-3.25})}.
\end{corollary}

\begin{conjecture}\label{conj-3.17}
Estimate (\ref{eqn-3.19}) and (\ref{eqn-3.22}) hold without $C'h^s$ in their right-hand expression.
\end{conjecture}

\begin{proof}[Proof of Theorem~\ref{thm-1.3}\ref{thm-1.3-iii}, \ref{thm-1.3-iv}]
Now standard scaling and setting $\varepsilon = C_0r^\beta$ imply  estimate (\ref{eqn-1.17})
(that is Theorem~\ref{thm-1.3}\ref{thm-1.3-iii}).

Furthermore, we see that as  $0<\alpha \le \frac{1}{2}$ one can skip the first term in the right-hand expression of this estimate (that is, remaining part of Theorem~\ref{thm-1.3}\ref{thm-1.3-iv}).
\end{proof}

 \section{Near singularity}
 \label{sect-3.3}

 Now we need to investigate zone $\{x\colon |x|\le   r_* = h^{1/(1-\alpha)-\delta}\}$ which we divide into non-semiclassical subzone
  $\{x\colon |x|\le \bar{r}=h^{1/(1-\alpha)}\}$ and intermediate subzone $\{x\colon \bar{r} \le |x| \le r_*\}$.

 \subsection{Non-semiclassical subzone}
 \label{sect-3.3.1}
 
 \begin{proposition}\label{prop-3.18}
As $|x|\le \bar{r}$, $|\y|\le \bar{r}$ and $\tau \le C_0\bar{r} ^{-2}$
\begin{gather}
|e_h(x,y,0)|\le C\bar{r}^{-d}
\label{eqn-3.26}
\shortintertext{and}
|e_h(x,y,0)- e^0_h(x,y,0)|\le C\bar{r}^{\beta-d}
\label{eqn-3.27}
\end{gather}
\end{proposition}

 \begin{proof}
 \begin{enumerate}[wide, label=(\roman*), labelindent=0pt]
\item\label{pf-3.18-i}
Let $E_h(\tau)$ be a spectral projector of $A_h$. Then $\|E_h(\tau) f\|_{\cL^2} \le \|f\|_{\cL^2}$. 
Then  $\| A_h E_h (\tau)\| \le |\tau| $ implies that  for $\tau \le c\bar{r}^{-2}$, $\hbar\asymp 1$ and $\|f\|_{\sL^2}=1$
\begin{gather}
\|D   E_h (\tau)f\|_{\sL^2} \le C\bar{r}^{-1}
\label{eqn-3.28}
\end{gather}
which in virtue Sobolev inequality  implies that 
\begin{enumerate}[wide, label=(\alph*), labelindent=10pt]
\item\label{pf-3.18-ia}
$d=1\implies \|E_h(\tau)f\|_{\sC}\le C\bar{r}^{-\frac{1}{2}}$, 
\item\label{pf-3.18-ib}
$d\ge 2\implies \|E_h(\tau)f\|_{\sL^p}\le C\bar{r}^{-1}$ with any $p<\infty$ for $d=2$ and $p=2d/(d-2)$ for $d\ge 3$.
\end{enumerate}

In the latter case using inequalities $\|V v\|_{\sL^q} \le C\| V\|_{\sL^s} \|v\|_{\sL^p}$ as $q^{-1}=s^{-1}+p^{-1}$ and $\|V\|_{\sL^s(B(0,r)} \le Cr^{d/s -2\alpha}$ as  $d/s > 2\alpha$ we upgrade (\ref{eqn-3.28}) to 
\begin{gather}
\|D ^2  E_h (\tau)f\|_{\sL^q} \le C\bar{r}^{-2}
\label{eqn-3.29}
\end{gather}
which in virtue Sobolev inequality  implies that  
\begin{enumerate}[wide, label=(\alph*), labelindent=10pt]
\setcounter{enumii}{2}
\item\label{pf-3.18-ic}
$\|E_h(\tau)f\|_{\sC}\le C\bar{r}^{-d/2}$ as \underline{either} $d=2$ \underline{or} $d\ge 3$ and $1/p<0$,
\item\label{pf-3.18-id}
$\|E_h(\tau)f\|_{\sL^{p'}}\le C\bar{r}^{-d/2+d/p'}$ as $d\ge 3$ and  $1/p'>0$ 
\end{enumerate}
with $1/p'=1/p-\delta>0$ with $\delta$ which does not depend on $p$. Repeating these arguments we decrease $1/p$ until it becomes negative and arrive to $\|E_h(\tau)f\|_{\sC}\le C\bar{r}^{-d/2}$ for all $f$ with $\|f\|_{\sL^2}=1$ which is equivalent to
\begin{gather}
\| e_h(x,y,\tau)\|_{\sL^2_y} \le C\bar{r}^{-d/2}\qquad \forall x\in B(0,\bar{r})
\label{eqn-3.30}
\end{gather}
which implies (\ref{eqn-3.27}).
 
\item\label{pf-3.18-ii}
Further, standard perturbation methods imply that  
\begin{gather*}
\|(E_h(\tau) - E^0_h(\tau)) f\|_{\cL^2(B(0,2\bar{r}))} \le  C\varepsilon \|f\|_{\cL^2}\quad \text{as\ \ } \supp (f)\subset B(0,2\bar{r})
\end{gather*}
 with $\varepsilon = \bar{r}^\beta$. 
 
 Then applying standard perturbation arguments\footref{foot-5} and combining with arguments of Part~\ref{pf-3.18-i} of the proof one can prove that
 \begin{enumerate}[wide, label=(\alph*), labelindent=10pt]
\setcounter{enumii}{4}
\item\label{pf-3.18-iie}
$ \|(E_h(\tau)-E^0_h(\tau)\bigr)f\|_{\sC}\le C\bar{r}^{\beta-\frac{1}{2}}$as \underline{either} $d=1$,
\item\label{pf-3.18-iif}
$\|\bigl(E_h(\tau)-E^0_h(\tau)\bigr)f\|_{\sC}\le C\varepsilon \bar{r}^{\beta-d/2}$ as \underline{either} $d=2$ \underline{or} $d\ge 3$ and $1/p<0$,
\item\label{pf-3.18-iig}
$\|\bigl(E_h(\tau)-E^0_h(\tau)\bigr)f\|_{\sL^{p'}}\le C\varepsilon \bar{r}^{\beta-d/2+d/p'}$ as $d\ge 3$ and  $1/p'>0$ 
\end{enumerate}
with $\varepsilon=C\bar{r}^\beta$. 

Repeating these arguments we decrease $1/p$ until it becomes negative and arrive to $\|\bigl(E_h(\tau)-E^0_h(\tau)\bigr)f\|_{\sC}\le C\bar{r}^{-d/2}$ for all $f$ with $\|f\|_{\sL^2}=1$ which is equivalent to
\begin{gather}
\| \bigl( e_h(x,y,\tau)- e^0_h(x,y,\tau)\bigr)\|_{\sL^2_y} \le C\bar{r}^{\beta -d/2}\qquad \forall x\in B(0,\bar{r})
\label{eqn-3.31}
\end{gather}
which together with \ref{eqn-3.30}) for both $e_h(x,y,\tau)$ and  $e^0_h(x,y,\tau)$ implies (\ref{eqn-3.28}).

We leave details to the reader.
\end{enumerate}
\end{proof}

\subsection{Intermediate subzone}
 \label{sect-3.3.2}

 \begin{remark}\label{rem-3.19}
 In the framework of Theorem~\ref{thm-1.3} in the intermediate zone $\{x\colon \bar{r}\le |x|\le r_*\}$
 \begin{enumerate}[wide, label=(\roman*), labelindent=0pt]
 \item\label{rem-3.19-i}
Estimates (\ref{eqn-1.13}) and (\ref{eqn-1.14}) with 
 $ r_*^{1+\alpha}\le  T\le \epsilon_0$ remain true.
 \item \label{rem-3.19-ii}
Estimates 
\begin{gather}
|e^\T _{T,h} (x,x,0)- e^\T _{T'',h} (x,x,0)|\le C h^{-d+s} r^{-\alpha d-(1-\alpha)s}
\label{eqn-3.32}
\end{gather}
and
\begin{multline}
|e^\T _{T,h} (x,x,0)- e^{0\,\T} _{T,h} (x,x,0) - e^\T _{T',h} (x,x,0)+ e^{0\,\T} _{T',h} (x,x,0)|\\
\le C h^{-d+s} r^{-\alpha d-(1-\alpha)s+\beta}
\label{eqn-3.33}
\end{multline}
with $T'=C_0r^{1+\alpha}$ follow from the similar estimates by rescaling.
 \item\label{rem-3.19-iii}
Estimates (\ref{eqn-1.15}) and (\ref{eqn-1.17}) with $T'=\epsilon_0r^{1+\alpha}$ follow from the same estimates by rescaling.
 \item\label{rem-3.19-iv}
 Furthermore, as $0<\alpha \le \frac{1}{2}$ corresponding terms in the right-hand expressions could be skipped.
  \item\label{rem-3.19-v}
  Estimates (\ref{eqn-1.18}) and (\ref{eqn-1.19}) with $T'=\epsilon_0r^{1+\alpha}$  follow from the same estimates by rescaling.
 \end{enumerate}
 \end{remark}

\bibliographystyle{amsplain}

\begin{thebibliography}{9999}

\bibitem[Arn]{arnold:classical}
V.~I.~Arnold.  \emph{Mathematical Methods of Classical Mechanics}.
Springer-Verlag, 1990.


\bibitem[Car1]{carleman:membrane}
T. Carleman.  \emph{Propri\'{e}tes asymptotiques des fonctions fondamentales des membranes
vibrantes}.
In C. R. 8-{\`{e}}me Congr. Math. Scand., Stockholm, 1934,
pages 34--44, Lund (1935).

\bibitem[Car2]{carleman:asymp}
T. Carleman. \emph{\"Uber die asymptotische {V}erteilung der {E}igenwerte partieller
Differentialgleichungen}.
Ber. Sachs. Acad. Wiss. Leipzig, 88:119--132 (1936).


\bibitem[Ivr1]{ivrii:precise}
V.~Ivrii.  \href{http://weyl.math.toronto.edu/victor2/preprints/LNM1100.djvu}{\emph{Precise Spectral Asymptotics for Elliptic Operators\/}.}
Lect. Notes Math. Springer-Verlag 1100 (1984).

\bibitem[Ivr2]{monsterbook}
V.~Ivrii. \href{http://www.math.toronto.edu/ivrii/monsterbook.pdf}{\emph{Microlocal Analysis, Sharp Spectral, Asymptotics and Applications}}.
Volume I. Semiclassical Microlocal Analysis and Local and Microlocal Semiclassical Asymptotics.

\bibitem[Ivr3]{OOD}
V.~Ivrii.  \href{https://arxiv.org/abs/2107.04807}{\emph{Pointwise spectral asymptotics out of the diagonal near boundary}}. ArXiv: 2107.04807, July 2021.

\bibitem[Ivr4]{OOD2}
V.~Ivrii.  \href{https://arxiv.org/abs/2107.04807}{\emph{Pointwise spectral asymptotics out of the diagonal near degeneration}}. ArXiv: 2107.04807, July 2021.

\bibitem[Ivr5]{Ivr5}V.~Ivrii. \href{https://arxiv.org/abs/1802.07524}{\emph{Spectral asymptotics for Dirichlet to Neumann operator in the domains with edges}}, ArXiv: 1802.07524, February 2018.


%

%


\end{thebibliography}

\end{document}